\documentclass{amsart}
\usepackage{graphicx}
\usepackage{verbatim}
\usepackage{fancyvrb}
\usepackage{xcolor}
\usepackage{amssymb}

\newcommand{\subf}[2]{%
	{\small\begin{tabular}[t]{@{}c@{}}
			#1\\#2
	\end{tabular}}%
}

\numberwithin{equation}{section} \numberwithin{figure}{section}
{\theoremstyle{remark} \newtheorem{remark}{Remark}}
\newtheorem{definition}{Definition}[section]
\newtheorem{proposition}{Proposition}[section]

\newtheorem{theorem}{Theorem}[section]
\newtheorem{corollary}{Corollary}[section]
\newtheorem{lemma}{Lemma}[section]

\newcommand\Laps{(-\Delta)^s}
\newcommand\Ns{\mathcal{N}_s}
\newcommand\pp{\partial}

\newcommand{\x}{\texttt{x}}
\newcommand{\eps}{\varepsilon}

\def\calT{\mathcal{T}}
\def\calF{\mathcal{F}}

\newcommand{\V}{\mathbb{V}}
\newcommand{\X}{\mathbb{X}}
\newcommand{\N}{\mathbb{N}}
\newcommand{\R}{\mathbb{R}}
\newcommand{\Rd}{{\mathbb{R}^d}}

\newcommand{\phii}{\varphi}
\DeclareMathOperator\supp{supp}

\newcommand{\fran}[1]{{\color{black}#1}}

\begin{document}

\title{Finite element approximation of fractional Neumann problems}

\author[J.P.~Borthagaray]{Juan Pablo~Borthagaray}
\address[J.P.~Borthagaray]{Departamento de Matem\'atica y
Estad\'istica del Litoral, Universidad de la Rep\'ublica, Salto,
Uruguay}
\email{jpborthagaray@unorte.edu.uy}

\author[F.M.~Bersetche]{Francisco M.~Bersetche}
\address[F.M.~Bersetche]{}
\email{bersetche@gmail.com}

\maketitle

\begin{abstract}
{In this paper we consider approximations of Neumann problems for the integral fractional Laplacian by continuous, piecewise linear finite elements. We analyze the weak formulation of such problems, including their well-posedness and asymptotic behavior of solutions. We address the convergence of the finite element discretizations and discuss the implementation of the method. Finally, we present several numerical experiments in one- and two-dimensional domains that illustrate the method's performance as well as certain properties of solutions.}
\end{abstract}

\section{Introduction and problem setting}

Let $\Omega \subset \R^d$ be a bounded Lipschitz domain, $s \in (0,1)$, $\alpha \ge 0$, and two given functions $f \colon \Omega \to \R$ and $g \colon \Omega^c \to \R$, where $\Omega^c = \Rd \setminus \overline\Omega$. In this work, we propose and study the convergence of a finite element scheme for the following problem: find $u \colon \R^d \to \R$ such that
\begin{equation} \label{eq:Neumann}
\left\lbrace \begin{array}{rl}
\Laps u + \alpha u = f & \mbox{in } \Omega, \\
\Ns u = g & \mbox{in } \Omega^c.
\end{array} \right.
\end{equation}

Above, $\Laps$ denotes the integral fractional Laplacian of order $s$,
\begin{equation}
\label{eq:defofLaps}
\Laps v(x) := C_{d,s} \mbox{ p.v. } \int_\Rd \frac{v(x)-v(y)}{|x-y|^{d+2s}} \, dy, \quad C_{d,s} :=  \frac{2^{2s}s\Gamma\left(s+\frac{d}{2}\right)}{\pi^{d/2}\Gamma(1-s)} .
\end{equation}
and $\Ns$ is the nonlocal Neumann operator
\begin{equation}
\label{eq:defofNs}
\Ns v (x) := C_{d,s} \int_\Omega \frac{v(x)-v(y)}{|x-y|^{d+2s}} \, d y.
\end{equation}

The fractional Laplacian $\Laps$ is a nonlocal operator: the evaluation of  $\Laps v (x)$ at any point $x \in \Omega$ involves the values of $v$ at the whole space $\R^d$. Therefore, boundary conditions in problem \eqref{eq:Neumann} need to be imposed on the complement of $\Omega$. The operator $\Ns$ depends on the domain $\Omega$, and can be interpreted as a nonlocal flux density between $\Omega^c$ and $\Omega$. We remark that there is no widely accepted definition of a Neumann condition for operator \eqref{eq:defofLaps} and refer the interested reader to \cite[Section 2.3.2]{review} and to \cite[Section 7]{DiRoVa17} for discussion on this aspect. The definition that we are using in this manuscript, that was proposed in \cite{DiRoVa17,Gunz2}, gives rise to the following integration by parts formula.

\begin{proposition}[integration by parts formula \cite{DiRoVa17, DuGuLeZh13}] Let $u, v \colon \R^d \to \R$ be  smooth enough functions, then
	\begin{equation}\label{eq:parts}
	\begin{aligned}
	\frac{C_{d,s}}{2} \iint_{(\R^d \times \R^d) \setminus (\Omega^c \times \Omega^c)} &  \frac{(u(x)-u(y))  (v(x)-v(y))}{|x-y|^{d+2s}} \, dx \, dy \,  \\ 
	& = \int_\Omega  v(x) (-\Delta)^su(x) \, dx + \int_{\Omega^c} v(x) \,  \Ns u(x) \, dx .
	\end{aligned} \end{equation}
\end{proposition}

To better illustrate the nonlocal derivative operator we are dealing with, let us mention a probabilistic interpretation for \eqref{eq:defofNs}. Consider the fractional heat problem with homogeneous Neumann condition on $\Omega$. Namely, suppose $u: \Omega \times [0,T] \to \R$ satisfies 
\begin{equation} \label{eq:Neumann_parabolico}
\left\lbrace \begin{array}{rl}
u_t + \Laps u = 0 & \mbox{in } \Omega \times (0,T], \\
\Ns u = 0 & \mbox{in } \Omega^c\times (0,T],\\
u(\cdot,0) = u_0 & \mbox{in } \Omega,
\end{array} \right.
\end{equation}
for some $T>0$, and $u_0 \in L^2(\Omega)$. In this context, the function $u$ can be understood as the probability density of the position of a particle moving randomly inside $\Omega$ according to a random walk with arbitrarily long jumps. The condition $\Ns u= 0$ refers to how the particle behaves when it jumps outside the domain: if it reaches a point $y \in \Omega^c$ then it may immediately come back to any point $x \in \Omega$, with a probability density proportional to $1/|x-y|^{d+2s}$. 

Problem \eqref{eq:Neumann} has a variational structure, which mimics the one for the standard Laplacian. Actually, solutions to \eqref{eq:Neumann} are critical points of the functional
\begin{equation} \label{eq:energy}
\calF (v) = \frac{C_{d,s}}{4} \iint_{(\R^d \times \R^d) \setminus (\Omega^c \times \Omega^c)}  \frac{|v(x)-v(y)|^2}{|x-y|^{d+2s}} \, dx \, dy + 
\frac\alpha2 \|v\|_{L^2(\Omega)}^2 - \int_\Omega f v - \int_{\Omega^c} g v.
\end{equation}
Such critical points are minima: in case $\alpha > 0$ there is a unique minimizer, while if $\alpha = 0$ minimizers are uniquely defined up to an additive constant, and one requires a compatibility condition on the data in order to guarantee the existence of solutions. The well-posedness of problem \eqref{eq:Neumann} in case $\alpha = 0$ is studied in \cite{DiRoVa17}. Here we shall focus on the case $\alpha > 0$, although the finite element scheme we propose can be straightforwardly adapted to the former case.

In recent years, finite element methods have been proposed and studied for a variety of equations involving the fractional Laplacian \eqref{eq:defofLaps}, such as Dirichlet \cite{ABB,AcosBort2017fractional,AcBoHe19,AiGl17,BBNOS18,BoLePa19}, time-fractional evolution \cite{ABBp}, phase field  \cite{AB_AC,ains1,HongWang}, optimal control \cite{AnKhWa19,AnVeWa20,BiHS19,DEGlOt19,GO}, and obstacle  \cite{BoLeSa20,BoNoSa18,BuGu19,GiSt19} problems. Most of these references consider either Dirichlet or periodic boundary conditions; reference \cite{AnKhWa19} deals with Neumann and Robin conditions, but does not address the convergence of finite element discretizations of such problems. The recent preprint \cite{CaJaWa20} studies Neumann problems closely related to \eqref{eq:Neumann} in one-dimensional domains by means of finite difference schemes. However, it proves convergence by assuming solutions to be of class $C^4$, and such a condition cannot be guaranteed in general.

Indeed, a crucial aspect in the numerical analysis of differential equations is the regularity of solutions. Reference \cite{AuFNRO20} studies the H\"older regularity of solutions to \eqref{eq:Neumann} whenever $\alpha = 0$ and $g \equiv 0$. However, to the best of our knowledge, there are no Sobolev regularity estimates for Neumann problems involving the integral fractional Laplacian in the literature. For that reason, we aim to prove the convergence of the finite element discretizations without assuming regularity of solutions (cf. Theorem \ref{thm:convergence} below). Nevertheless, in our numerical experiments we have computed convergence rates whenever explicit solutions were available.

Throughout the paper we denote by $C$ any nonessential constant, and by $A \simeq B$ we mean that $A \le CB$ and $B \le CA$. Whenever we want to express the dependence of $C$ on $A$, we write it as $C_A$.

This manuscript has been organized in the following way. In Section \ref{sec:weakform} we set the weak formulation of problem \eqref{eq:Neumann}, prove a nonlocal trace theorem for functions in a suitable variational space, and derive asymptotic estimates for solutions. Section \ref{sec:FE} is devoted to the description of the finite element method, while its convergence is treated in Section \ref{sec:convergence} along with several interpolation estimates. Section \ref{sec:experiments} exhibits several numerical experiments. Not only do these experiments illustrate the convergence of the finite element discretizations but also their capability of capturing certain properties of solutions, such as limits at infinity and exponential convergence to the mean of the initial datum for the fractional heat equation with homogeneous Neumann conditions. Finally, Appendices \ref{app:g_construida} and \ref{sec:implementation} offer some details about the implementation of the method.

\section{Weak Formulation}
\label{sec:weakform}

The integration by parts formula \eqref{eq:parts} allows us to set a weak formulation for problem \eqref{eq:Neumann}. For that purpose, we first need to define a suitable variational space.

\begin{definition}[variational space] We set
	\begin{equation*} \label{eq:defofV}
	\V := \{ v : \R^d \to \R \mbox{ measurable } \colon \| v \|_{\V} < \infty  \},
	\end{equation*}
	where 
	\begin{equation} \label{eq:norm_V}
	\| v \|_{\V} := \left( \| \fran{v} \|_{L^2(\Omega)}^2 + | \fran{v} |_{\X}^2  \right)^{1/2},
	\end{equation}
	and
	\begin{equation} \label{eq:norm_X}
	| \fran{v} |_{\X} := \left( \frac{C_{d,s}}{2} \iint_{(\R^d \times \R^d) \setminus (\Omega^c \times \Omega^c)}  \frac{|\fran{v}(x)-\fran{v}(y)|^2}{|x-y|^{d+2s}} \, dx \, dy \right)^{1/2}.
	\end{equation}
\end{definition}

The space $\V$ introduced above \fran{is motivated by \cite{DiRoVa17} and coincides with the space $H^s_{\Omega,0}$ in that reference}. In particular, from \cite[Proposition 3.1]{DiRoVa17}, it follows that $\V$ is a Hilbert space. We shall denote by $\langle \cdot , \cdot \rangle_{\X} \colon \X \times \X \to \R$ the bilinear form
\begin{equation*} \label{eq:def-pintX}
\langle u , v \rangle_{\X} := \frac{C_{d,s}}{2} \iint_{(\R^d \times \R^d) \setminus (\Omega^c \times \Omega^c)}  \frac{(u(x)-u(y))  (v(x)-v(y))}{|x-y|^{d+2s}} \, dx \, dy
\end{equation*} 
and by $(\cdot, \cdot)_{L^2(\Omega)}$ the standard inner product in $L^2(\Omega)$ or any duality pairing using $L^2(\Omega)$ as pivot space. The variational space $\V$ is also related to fractional-order Sobolev spaces; when necessary, we shall adopt the notation from \cite{ABB} regarding such spaces. \fran{We point out that, unlike the fractional Sobolev space $H^s(\Omega)$, the space $\mathbb{V}$ takes into account interactions between $\Omega$ and $\Omega^c$; moreover, unlike the space $H^s(\Rd)$, the space $\mathbb{V}$ does not take into account interactions between $\Omega^c$ and $\Omega^c$.}

Using the variational space and notation we have just introduced and \eqref{eq:parts}, the weak formulation of our problem reads as follows: find $u \in \V$ such that
\begin{equation} \label{eq:weak_form}
\begin{aligned}
\langle u , v \rangle_{\X} + \alpha (u,v)_{L^2(\Omega)} =
(f,v)_{L^2(\Omega)} + (g,v)_{L^2(\Omega^c)} \quad \forall v \in \V.
\end{aligned} \end{equation}

In order to study the well-posedness of this weak formulation, we need to make sense of the right hand side in \eqref{eq:weak_form}. Specifically, we need some control of the behavior in $\Omega^c$ of functions in $\V$; we shall accomplish this by proving an inequality in the spirit of a nonlocal trace theorem. 

It seems clear from \eqref{eq:norm_V} and \eqref{eq:norm_X} that one cannot hope to have control of \fran{the smoothness of a function within $\Omega^c$} in terms of \fran{its} $\V$-norm. Thus, one might try instead to bound a $L^p(\Omega^c)$-norm in terms of the $\V$-norm. However, because $|\Omega^c| = \infty$ and any constant function belongs to $\V$, it is apparent that one cannot expect the inequality 
$
\| v \|_{L^p(\Omega^c)} \le C \| v \|_\V
$
to hold for any $1\le p < \infty$.

\begin{remark}[blow up at infinity] \label{rmk:blow-up}
	Given a fixed number $R >\mbox{diam}(\Omega)$ let us define
	\begin{equation} \label{eq:defofLambda}
	\Lambda_R := \{ x \in \R^d \colon d(x,\pp\Omega) \le R \} 
	\end{equation}
	and consider a smooth, locally bounded function $v \colon \R^d \to \R$ such that
	\begin{equation} \label{eq:defofv}
	v \equiv 1 \mbox{ in } \Lambda_R , \qquad v(x) \simeq |x|^\alpha  \mbox{ in } \Lambda_R^c,
	\end{equation}
	for some $\alpha \in (0,s)$. Then, exploiting that
	\begin{equation} \label{eq:decay-Lambda-Omega}
	\int_\Omega \frac{1}{|x-y|^{d+2s}} dx \simeq |y|^{-d-2s} \quad \mbox{for  } y \in \Lambda_R^c
	\end{equation}
	and the equivalence 
	\[
	\int_{\Lambda_R^c} |y|^{-d - 2(s-\alpha)} dy \simeq  R^{- 2(s-\alpha)},
	\]
	which follows by integration in polar coordinates, we obtain
    \[
	|v|_\X^2 \le \fran{C \left( 1 +  \iint_{\Omega \times \Lambda_R^c} \frac{|v(y)|^2}{|x-y|^{d+2s}} dy dx \right)} \le 
	C \left( 1 +  \int_{\Lambda_R^c} |y|^{-d - 2(s-\alpha)} dy \right) < \infty.
	\]
	In consequence, this function satisfies $v \in \V$, \fran{although it is unbounded at infinity.}
\end{remark}

It seems therefore natural to consider weighted norms, that allow functions to have some growth at infinity. We consider the following spaces.

\begin{definition}
	Let $p \in \fran{[1, \infty]}$ and $\gamma \in \R$. Then, we define the space
	\[
	L^p_{\gamma}(\Rd) := \left\{ v \colon \Rd \to \R  \mbox{ measurable } \colon  \| v \|_{L^p_{\gamma}(\Rd)} < \infty \right\}, 
	\]
	where
	\[
	\| v \|_{L^p_{\gamma}(\Rd)} := \left\lbrace \begin{array}{ll} \left(\int_\Rd \frac{|v(x)|^p}{1+|x|^{d+\gamma p}} \, dx \right)^{1/p} & \mbox{if } 1 \le p < \infty , \\
	\sup_{x \in \Rd} \frac{|v(x)|}{1+|x|^\gamma} & \mbox{if } p = \infty.
	\end{array} \right.
	\]
\end{definition}

\begin{remark}[relations between the spaces $L^p_\gamma(\Rd)$] From the definition above, it follows immediately that $L^p_{\gamma_1}(\Rd) \subset L^p_{\gamma_2}(\Rd)$ if $\gamma_1 \le \gamma_2$. Also, an application of H\"older's inequality gives that if $p_1 > p_2$ and $\gamma_1 < \gamma_2$, then $L^{p_1}_{\gamma_1}(\Rd) \subset L^{p_2}_{\gamma_2}(\Rd)$.
\end{remark}

Let us focus on the exponent $p=2$. Remark \ref{rmk:blow-up} guarantees that, in order to have $\V \subset L^2_\gamma(\Rd)$, the weight exponent $\gamma$ cannot be too small. We now make more precise such an assertion.

\begin{lemma}[admissible exponents]
	If $\gamma < s$, then $\V \not\subseteq L^2_\gamma(\Rd)$.
\end{lemma}
\begin{proof}
	Let $\gamma < s$ and set $\alpha = \gamma < s$. We take a function $v$ as in \eqref{eq:defofv}, which satisfies $v \in \V$. However, our choice of $\alpha$ trivially yields
	\[
	\fran{\| v \|^2_{L^2_\gamma(\Rd)}} \ge C \int_{\Lambda_R^c} |x|^{-d} dx.
	\]
	Because the integral in the right hand side above is divergent, $v \notin L^2_\gamma(\Rd)$.
\end{proof}

The following trace-type inequality asserts that the value $\gamma = s$ is indeed critical.

\begin{proposition}[trace-type inequality] \label{prop:trace}
	Let $\gamma \ge s$. There exists a constant $C > 0$ such that, for all $v \in \V$,
	\begin{equation} \label{eq:trace}
	\| v \|_{L^2_\gamma(\Rd)} \le C \| v \|_{\V} .
	\end{equation}
	Thus, the embedding $\V \subset L^2_\gamma(\Rd)$ is continuous for all $\gamma \ge s$.
\end{proposition}
\begin{proof}
	We split $\R^d = \Lambda_R \cup \Lambda_R^c$, and compute the $L^2$-norms on each subset separately. Let $x \in \Lambda_R$. Given $y \in \Omega$, because $|x-y| \le 3R$ we can write
	\[
	|v(x)|^2 \le 2 (3R)^{d+2s} \frac{|v(x) - v(y)|^2}{|x-y|^{d+2s}} + 2 |v(y)|^2.
	\]
	We integrate the inequality above over $\Lambda_R \times \Omega$ to obtain
	\[
	|\Omega| \int_{\Lambda_R} |v(x)|^2 dx \le C \left( R^{d+2s} \iint_{\Lambda_R \times \Omega}\frac{|v(x) - v(y)|^2}{|x-y|^{d+2s}} dy dx + R^d \int_{\Omega} |v(y)|^2 dy \right).
	\]
	
	Because $1+|x|^{d+2\gamma} \ge 1$, we deduce that
	\begin{equation} \label{eq:near_bound}
	\int_{\Lambda_R} \frac{|v(x)|^2}{1+|x|^{d+2\gamma}} dx
	\le C \left( R^{d+2s} |v|^2_\X + R^d \| v \|^2_{L^2(\Omega)} \right).
	\end{equation}
	
	On the other hand, because $\gamma \ge s$, if $x \in \Lambda_R^c$ then we have
	\begin{equation} \label{eq:decay_measure}
	\int_\Omega \frac{dy}{|x-y|^{d+2s}} \simeq (1+|x|^{d+2s})^{-1} \ge C (1+|x|^{d+2\gamma})^{-1} R^{2(\gamma - s)}.
	\end{equation}
	Therefore, we obtain
	\[ \begin{split}
	\int_{\Lambda_R^c} \frac{|v(x)|^2}{1+|x|^{d+2\gamma}} dx & \le C R^{2s - 2\gamma} \iint_{\Lambda_R^c \times \Omega} \frac{|v(x)|^2}{|x-y|^{d+2s}} dy dx  \\ 
	& \le  C R^{2s - 2\gamma}\left( \iint_{\Lambda_R^c \times \Omega} \frac{|v(x) - v(y)|^2}{|x-y|^{d+2s}} dy dx + \iint_{\Lambda_R^c \times \Omega} \frac{|v(y)|^2}{|x-y|^{d+2s}} dy dx \right).
	\end{split} \]
	The first integral in the right hand side above is bounded by $|v|^2_\X$. In order to bound the second one, we observe that
	\begin{equation} \label{eq:decay-Omega-Lambda}
	\int_{\Lambda_R^c} \frac{dx}{|x-y|^{d+2s}} \simeq R^{-2s} \quad \mbox{for } y \in \Omega.  
	\end{equation}
	Using this identity, we immediately get
	\[
	\iint_{\Lambda_R^c \times \Omega} \frac{|v(y)|^2}{|x-y|^{d+2s}} dy dx \le C R^{-2s} \| v \|^2_{L^2(\Omega)}.
	\]
	Thus, we have shown that
	\[
	\int_{\Lambda_R^c} \frac{|v(x)|^2}{1+|x|^{d+2\gamma}} dx  \le C \left(  R^{2s - 2\gamma} |v|_\X^2 + R^{-2\gamma} \| v \|^2_{L^2(\Omega)} \right) ,
	\]
	and combining this estimate with \eqref{eq:near_bound}, we conclude that \eqref{eq:trace} holds.
\end{proof}

The trace-type inequality we have just proved yields the boundedness of the operator $\V \ni v \mapsto (g, v)_{L^2(\Omega^c)}$, which in turn gives rise to the well-posedness of the weak formulation. \fran{Let us denote by $[L^2_{\gamma}(\Omega^c)]'$ the dual space to $L^2_{\gamma}(\Omega^c)$.}
We shall assume that the nonlocal flux density $g$ \fran{belongs to $[L^2_{\gamma}(\Omega^c)]'$ for some $\gamma \ge s$, so that it} satisfies the condition
\begin{equation}\label{eq:cond-g}
\fran{\| g \|_{[L^2_{\gamma}(\Omega^c)]'}^2 =} \int_{\Omega^c} |g(x)|^2 (1+|x|^{d+2\gamma}) \, dx. < \infty
\end{equation}
Combining this hypothesis with Proposition \ref{prop:trace} gives
\begin{equation} \label{eq:boundedness-g}
\int_{\Omega^c} g(x) v(x) \, dx \le \fran{\| g \|_{[L^2_{\gamma}(\Omega^c)]'} \| v \|_{\V} .}
\end{equation}

\begin{lemma}[well-posedness] \label{lem:well-posedness}
	Let $s \in (0,1)$, $\alpha > 0$, $f \in L^2(\Omega)$ and \fran{$g \in [L^2_{s}(\Omega^c)]'$, so} that \eqref{eq:cond-g} holds with $\gamma = s$. 
	Then, there exists a unique $u \in \V$ that solves the weak problem \eqref{eq:weak_form}.
\end{lemma}
\begin{proof}
	The proof follows immediately by the Lax-Milgram lemma. On the one hand, because $\alpha > 0$ the bilinear functional
	\[
	\V \times \V \ni (u,v) \mapsto \langle u , v \rangle_{\X} + \alpha (u,v)_{L^2(\Omega)}
	\]
	is trivially continuous and coercive.
	
	On the other hand, the continuity of the map
	\[
	\V \ni v \mapsto  (f,v)_{L^2(\Omega)} + (g,v)_{L^2(\Omega^c)}
	\]
	follows because $|v|_{H^s(\Omega)} \le \|v\|_{\V}$ and by \eqref{eq:boundedness-g}:
	\[
	\left| \int_\Omega f(x) v(x) \, dx + \int_{\Omega^c} g(x) v(x) \, dx \right| \le \left(\|f \|_{L^2(\Omega)} + \fran{\| g \|_{[L^2_{s}(\Omega^c)]'}} \right) \| v \|_\V .
	\]
\end{proof}

\begin{remark}[energy minimizer] Using standard arguments, one can show that $u \in \V$ solves \eqref{eq:weak_form} if and only if $u$ is a critical point of the energy $\calF$ in \eqref{eq:energy} and, in turn, that such an energy admits a unique minimizer:
	\[
	\calF(v) = \calF(u) + \frac12 |v-u|_\X^2 + \frac{\alpha}{2} \| v-u\|_{L^2(\Omega)}^2, \quad \forall v \in \V.
	\]
\end{remark}

\begin{remark}[case $\alpha = 0$]
	Naturally, in case $\alpha = 0$ one requires the compatibility condition 
	\[
	\int_\Omega f = - \int_{\Omega^c} g
	\]
	to guarantee the well-posedness of the weak problem, whose solution is unique up to an additive constant. We refer to \cite[Theorem 3.9]{DiRoVa17} for details. We point out that such a Theorem has the less restrictive decay hypothesis $g \in L^1(\Omega^c)$, but it additionally requires the existence of some $\psi$ of class $C^2$ such that $\Ns \psi = g$ in $\Omega^c$.
\end{remark}

\subsection{Decay of solutions} 
When performing finite element discretizations of \eqref{eq:weak_form}, we shall need to truncate $\Omega^c$ and compute solutions over a family of computational domains $\{ \Lambda_H \}$ with finite diameter. We shall allow the finite element solutions not to vanish over $\Lambda_H^c$ but rather to be constant on this set. While this adds an additional degree of freedom in our computations, it gives an improvement in the approximation of solutions (cf. Remark \ref{rmk:average_uh} below). 

This is particularly useful if the exact solution was known to be bounded at infinity, which a priori may not be the case. As we discussed in Remark \ref{rmk:blow-up}, functions in $\V$ may blow up like $|x|^\alpha$ for $\alpha \in (0,s)$. Because $u \in \V$ is the solution of \eqref{eq:weak_form}, 
one can prove further decay of $u$ by assuming further decay on the flux density $g$.

\begin{proposition}[decay of solutions] \label{prop:further_decay}
	Let $s \in (0,1)$, $\alpha > 0$, $f \in L^2(\Omega)$ and \fran{$g \in [L^2_{s+\beta}(\Omega^c)]'$} for some $\beta \in (0,s)$. 
	Then, the unique solution $u \in \V$ of \eqref{eq:weak_form} belongs to the space $L^2_{s-\beta}(\R^d)$, and it satisfies
	\[
	\| u \|_{L^2_{s-\beta}(\R^d)} \le \fran{ C \left( \| g \|_{[L^2_{s+\beta}(\Omega^c)]'} + \| u \|_{\V} \right). }
	\]
\end{proposition}
\begin{proof}
	Let $R > 0$. Using the notation \eqref{eq:defofLambda} and taking into account the first part of the proof of Proposition \ref{prop:trace}, we only need to estimate $\left\| \frac{u}{\sqrt{1+|\cdot|^{d+2(s-\beta)}}} \right\|_{L^2(\Lambda_R^c)}$. For that purpose, we exploit that for a.e. $x \in \Lambda_R^c$ it holds
	\[
	g(x) = \Ns u (x) = C_{d,s} \int_\Omega \frac{u(x) - u(y)}{|x-y|^{d+2s}} \, dy,
	\]
	and therefore
	\begin{equation}
	\label{eq:dato_afuera}
	u(x) \, C_{d,s}  \int_\Omega \frac{1}{|x-y|^{d+2s}} \, dy = g(x) + C_{d,s} \int_\Omega \frac{u(y)}{|x-y|^{d+2s}} \, dy.
	\end{equation}
	
	We use \eqref{eq:decay_measure}, the Cauchy-Schwarz inequality to obtain
	\begin{equation*} \label{eq:decay-u}
	\frac{|u(x)|}{1+|x|^{d+2s}}  \le C \left( |g(x)| + \frac{\| u \|_{L^2(\Omega)}}{1+|x|^{d+2s}} \right),
	\end{equation*}
	and multiplying both sides by $1+|x|^{d/2+s+\beta}$, taking squares and integrating over $\Lambda_R^c$, we deduce
	\[
	\int_{\Lambda_R^c} \frac{|u(x)|^2}{1+|x|^{d+2(s-\beta)}} dx \le C \left( 
\fran{\| g \|_{[L^2_{s+\beta}(\Lambda_R^c)]'}^2}
	+ R^{-2(s - \beta)} \| u \|_{L^2(\Omega)}^2 \right). 
	\]
	The result follows.
\end{proof}

\begin{remark}[optimality] \label{rmk:sol-constante}
	A simple example shows that the open-endedness of the range $\beta < s$ in Proposition \ref{prop:further_decay} is optimal. Indeed, assume that $g \equiv 0$ and $f \equiv \alpha$ in \eqref{eq:weak_form}. Then, the solution to such a problem is $u \equiv 1$, that satisfies $u \in \cap_{\beta < s} L^2_{s-\beta}(\Omega^c)$ but $u \notin L^2_0(\Omega^c)$.
\end{remark}

\begin{corollary}[Neumann conditions with strong decay] \label{cor:strong-decay}
	Let $s \in (0,1)$, $\alpha \ge 0$, $f \in L^2(\Omega)$ and $g$ be such that 
	\begin{equation} \label{eq:strong-decay-g}
	g(x) |x|^{d+2s} \to 0 \quad \mbox{as } |x| \to \infty.
	\end{equation} 
	Then, the unique solution $u \in \V$ of \eqref{eq:weak_form} satisfies
	\[
	\lim_{|x| \to \infty} u(x) = \frac{1}{|\Omega|} \int_{\Omega} u  = \frac1{\alpha|\Omega|}\left(\int_\Omega f + \int_{\Omega^c} g \right).
	\]
\end{corollary}
\begin{proof}
	We exploit formula \eqref{eq:dato_afuera}. In first place,  arguing as in \cite[Proposition 3.13]{DiRoVa17} one derives that 
	\[
	\lim_{|x| \to \infty} \frac{\int_\Omega \frac{u(y)}{|x-y|^{d+2s}} \, dy}{\int_\Omega \frac{1}{|x-y|^{d+2s}} \, dy} = \frac{1}{|\Omega|} \int_{\Omega} u .
	\]
	Additionally, from the decay hypothesis \eqref{eq:strong-decay-g} and \eqref{eq:decay_measure}, we have
	\[
	\lim_{|x| \to \infty} \frac{g(x)}{\int_\Omega \frac{1}{|x-y|^{d+2s}} \, dy } = 0.
	\]
	Finally, using the test function $v \equiv 1$ in \eqref{eq:weak_form} we deduce that $\frac{1}{|\Omega|} \int_{\Omega} u = \frac1{\alpha|\Omega|}(\int_\Omega f + \int_{\Omega^c} g)$.
\end{proof}

\begin{remark}[Neumann conditions with weaker decay]\label{rmk:asintotico}
	In a similar fashion as in Corollary \ref{cor:strong-decay}, it follows that if $g \ge 0$ is such that $g(x) |x|^{d+2s} \to \infty$ as $|x| \to \infty$, then the solution $u \in \V$ of \eqref{eq:weak_form} verifies $u(x) \to + \infty$ as $|x| \to \infty$. More in general, if $g(x) |x|^{d+2s} \to \kappa$ as $|x| \to \infty$ for some $\kappa \in \R$, then
	\[
	u(x) \to \frac{\kappa}{C_{d,s}|\Omega|} +  \frac{1}{|\Omega|} \int_{\Omega} u, \quad \mbox{as } |x| \to \infty.
	\]
\end{remark}

\subsection{Interior regularity}
Besides decay of solutions at infinity, another important aspect we need to take into account is their interior regularity within $\Omega$. 
We make use of a local regularity estimate from \cite[Theorem 2.1]{CozziM_2017}. Such a result requires the condition $u \in L^1_{2s}(\Rd)$; because of the continuity of the embedding $L^2_s(\Rd) \subset L^1_{2s}(\Rd)$, this assumption holds whenever the Neumann datum verifies \fran{$g \in [L^2_{s}(\Omega^c)]'$.}

\begin{theorem}[interior regularity] \label{thm:interior_regularity}
	Under the same conditions as Lemma \ref{lem:well-posedness}, the unique solution $u \in \V$ of \eqref{eq:weak_form} satisfies $u \in \cap_{\eps > 0} H^{2s-\eps}_{loc}(\Omega)$, and for every $\eps > 0$ and $\Omega' \Subset \Omega$,
	\[
	\| u \|_{H^{2s-\eps}(\Omega')} \le C \left(\| f \|_{L^2(\Omega)} + \| u \|_{L^2(\Omega)} + \| u \|_{L^1_{2s}(\R^d)} \right).
	\]
\end{theorem}

\section{Discretization} \label{sec:FE}
We approximate \eqref{eq:weak_form} by means of the finite element method. For that purpose, we consider a mesh-size number $h>0$ and, for $H=H(h)>0$ we take a computational domain $\Lambda_H$ according to \eqref{eq:defofLambda}. We consider admissible triangulations $\calT_h$ of $\Lambda_H$, which we assume that mesh $\overline \Omega$ exactly. Additionally, the family $\{ \calT_h \}$ is set to be shape-regular, namely,
\[
\sigma := \sup_{h>0} \max_{T \in \calT_h} \frac{h_T}{\rho_T} <\infty,
\]
where $h_T = \mbox{diam}(T)$ and $\rho_T $ is the diameter of the largest ball contained in $T$. As usual, the subindex $h$ denotes the mesh size, $h = \max_{T \in \calT_h} h_T$; moreover, we take elements to be closed sets.

We make use of continuous, piecewise linear functions over $\calT_h$. Let $\mathcal{N}_h$ be the set of vertices of $\calT_h$, $N$ be its cardinality, and $\{ \phii_i \}_{i=1}^N$ the standard piecewise linear Lagrangian basis, with $\phii_i$ associated to the node $\x_i \in \mathcal{N}_h$. In order to better capture the behavior of solutions at infinity, we additionally make use of constant functions over $\Lambda^{c}_H$. That is, we define $\phii_{N+1} := \chi_{\Lambda^c_H}$ and set
\begin{equation*} \label{eq:FE_space}
\mathbb{V}_h :=  \left\{ v_h \in C_0(\Lambda_H) \colon v_h = \sum_{i=1}^{N+1} v_i \varphi_i \right\}.
\end{equation*}

We emphasize that, in principle, the computational-domain size $H$ could be related to the mesh size number $h$. To prove the convergence of the finite element scheme we need $H \to \infty$ when $h \to 0$.

With the notation we just have defined, we seek a function $u_h \in \V_h$ such that
\begin{equation} \label{eq:weak_discrete}
\begin{aligned}
\langle u_h , v_h \rangle_{\X} + \alpha (u_h,v_h)_{L^2(\Omega)} =
(f,v_h)_{L^2(\Omega)} + (g,v_h)_{L^2(\Omega^c)} 
\end{aligned} \end{equation}
for all $v_h \in \V_h$. If we set $u_h = \sum_{i=1}^{N+1} U_i \phii_i$, we can write the weak formulation as a linear system of equations,
\begin{equation} \label{eq:disc_system}
\left(K + \alpha M \right) U = F + G, 
\end{equation}
where
\[
K_{ij} = \langle \phii_i , \phii_j \rangle_{\X} , \quad M_{ij} = (\phii_i,\phii_j)_{L^2(\Omega)}, \quad
F_j =  (f,\phii_j)_{L^2(\Omega)} \quad G_j = (g,\phii_j)_{L^2(\Omega^c)} .
\]
The stiffness matrix $K$ is symmetric and semidefinite positive, and because $\alpha > 0$ the matrix $\alpha M$ is symmetric and definite positive. Therefore, the system \eqref{eq:disc_system} admits a unique solution.

Since we are using discrete functions over $\overline{\Lambda_H}$ and a constant basis function on $\Lambda_H^c$, our discretizations are conforming: it holds that $\V_h \subset \V$ for all $h > 0$. By Galerkin orthogonality, we immediately deduce that
\[
| u - u_h |_\X^2 + \alpha \| u - u_h \|_{L^2(\Omega)}^2 = \min_{v_h \in \V_h} \left( | u - v_h |_\X^2 + \alpha \| u - v_h \|_{L^2(\Omega)}^2 \right), 
\]
from which the estimate
\begin{equation} \label{eq:best-approximation}
\| u - u_h \|_\V \le \max \{ \sqrt\alpha, \sqrt\alpha^{-1} \} \min_{v_h \in \V_h} \| u - v_h \|_\V.
\end{equation}
follows.

\begin{remark}[averages of finite element solutions] \label{rmk:average_uh}
	Because the constant function $v_h \equiv 1$ belongs to the discrete spaces $\V_h$ for all $h,H>0,$ we may use them as test functions in \eqref{eq:weak_discrete}. Therefore, it follows that the finite element solutions have the same averages over $\Omega$ as the solutions of \eqref{eq:weak_form},
	\[
	\frac{1}{|\Omega|} \int_{\Omega} u_h = \frac1{\alpha|\Omega|}\left(\int_\Omega f + \int_{\Omega^c} g\right) = \frac{1}{|\Omega|} \int_{\Omega} u .
	\]
	We point out that this property would not hold in general if we had not included the additional degree of freedom corresponding to $\phii_{N+1}$.
\end{remark}

\section{Interpolation and Convergence}
\label{sec:convergence}

Here we study the convergence of the finite element scheme proposed in Section \ref{sec:FE}. For that purpose, we first introduce a quasi-interpolation operator and analyze its stability and approximation properties. We afterwards combine these results with the best approximation properties of the finite element solution to prove the convergence of the method for locally bounded solutions but without any additional smoothness assumption.

\subsection{Interpolation}
We define the star of a set $A \in \Omega$ by
$$
S^1_A := \bigcup \left\{ T \in \calT_h \colon T \cap A \neq \emptyset \right\}.
$$
Given $T \in \calT_h$, the star $S^1_T$ of $T$ is the first ring of $T$. Recursively, we define the higher-order rings of $T$: $S^{k+1}_{T} = S^1_{S^k_T}$, $k \in \N$.  The star of the node $\x_i \in \mathcal{N}_h$ is $S_i := \mbox{supp}(\varphi_i)$.
We denote by $B_i$ the maximal ball, centered at $\x_i$, and contained in $S_i$. If $\rho_i$ is the radius of $B_i$, and $h_i = \mbox{diam}(S_i)$ by shape regularity of the mesh we have the equivalences $\rho_i \simeq h_i \simeq h_T$, for all $T \subset S_i$.

A detailed proof of the following observation, which is due to Faermann \cite{Faermann}, can be found in \cite[Lemma 3.2]{BoLeNo20}.

\begin{lemma}[symmetry]\label{lemma:symmetry}
	For any $v,w\in L^1(\Lambda)$, and $\rho:\R^+\to\R^+$ bounded, there holds
	\[
	\sum_{T\in\calT_h \colon T \cap \Lambda \neq \emptyset} \int_T \int_{(S_T^1)^c  \cap \Lambda}  v(y) \, w(x) \, \rho(|x-y|) dy dx = 
	\sum_{T\in\calT_h \colon T \cap \Lambda \neq \emptyset} \int_T \int_{(S_T^1)^c \cap\Lambda}  v(x) \, w(y) \, \rho(|x-y|) dy dx.
	\]
\end{lemma}

We split the mesh nodes into two disjoint sets, consisting of either vertices in $\overline{\Omega}$ and in $\Omega^c$,
\[
\mathcal{N}_h^\circ = \left\{ \x_i \colon \x_i \in \overline \Omega \right\}, \qquad  \mathcal{N}_h^c = \left\{ \x_i \colon \x_i \in \Omega^c \right\} = \left\{ \x_i \colon S_i \subset \overline{\Omega^c} \right\} .
\]

We shall construct a quasi-interpolation (averaging) operator that, within $\overline\Omega$, considers averages over $\Omega$ only. For that purpose, given a mesh node $\x_i$, we define the region
\[
R_i := \left\lbrace \begin{array}{rl}
B_i & \mbox{if } \x_i \in \R^d \setminus \pp\Omega, \\
B_i \cap \Omega & \mbox{if } \x_i \in \pp\Omega.
\end{array} \right.
\]
\fran{This definition guarantees that the broken quasi-interpolation operator defined below only takes averages within $\Omega$ for nodes in $\mathcal{N}_h^\circ$ and within $\Omega^c$ for nodes in $\mathcal{N}_h^c$.}
We remark that shape regularity implies $|R_i| \simeq h_T^d$ for all $T \subset S_i$.

\begin{definition}[quasi-interpolation operator]
	Let the broken quasi-interpolation operator $I_h : L^1(\Omega) \to \V_h$ be defined by
	\[
	I_h v = \sum_{\x_i \in \mathcal{N}_h} \left( \frac1{|R_i|} \int_{R_i} v(x) dx \right) \phii_i.
	\]
\end{definition}

We remark that the definition above implies that $I_h v \equiv 0$ over the non-meshed region $\Lambda_H^c$. As long as one takes $H \to \infty$ as $h \to 0$, one can guarantee that the interpolation error tends to zero.

The operator $I_h$ is based on the positivity-preserving operator from \cite{ChenNochetto}; indeed, it coincides with such an operator everywhere except in the discrete boundary layer
\[
\{ T \in \calT_h \colon T \cap \pp\Omega \neq \emptyset \} .
\]
We shall therefore exploit some of the properties of that operator documented in \cite{BoNoSa18, ChenNochetto}. For instance, because for every $\x_i \in \Omega$ the ball $B_i$ is symmetric with respect to $\x_i$, the operator $I_h$ satisfies
\begin{equation*}
\label{eq:linear}
I_h v(\x_i) = v(\x_i), \quad \forall v \in P_1(B_i),
\end{equation*}
where by $P_1(E)$ we denote the space of polynomials of degree one over the set $E$. However, this operator is not a projection: in general $I_h v_h \neq v_h$ for $v_h \in \mathbb{V}_h$ even in the interior of the domain \cite{NochettoWahlbin}. 

Let $T \in \calT_h$ and consider its modified ring of order $k \in \N$,
\begin{equation*} \label{eq:patch-T}
\widetilde{S}^k_T = 
\left\lbrace \begin{array}{rl}
S^k_T & \mbox{if }  T \subset \Omega^c, \\
S^k_T \cap \Omega & \mbox{if } T \subset \Omega.
\end{array} \right.
\end{equation*}
Using standard arguments, one can prove the following estimates: 

\begin{equation} \label{eq:L2-interpolation-error}
\| v - I_h v \|_{L^2(T)} \le C h^t |v|_{H^t(\widetilde{S}^1_T)},
\end{equation}
\begin{equation} \label{eq:Hs-interpolation-error}
\int_T \int_{S^1_T}   \frac{ |(v-I_h v) (x) - (v-I_h v)(y) |^2}{|x-y|^{d+2s}} dy dx \le C h^{2(t-s)} |v|_{H^t(\widetilde{S}^2_T)}^2,
\end{equation}

These interpolation estimates are satisfactory to deal with functions that are locally smoother than $H^s$. However, we only know the solution $u$ of our problem to have such a regularity in the interior of the domain (cf. Theorem \ref{thm:interior_regularity}). The method we shall pursue to prove the convergence of $I_hu$ towards $u$ as $h\to 0$ relies on the stability of $I_h$. We now develop various stability estimates that will be employed to prove the convergence of our finite element scheme.

\begin{lemma}[stability w.r.t. to averages]
	\label{lem:stabIh}
	Let $s \in (0,1)$ and $T, T' \in \calT_h$. There is a constant $C$, depending only on the dimension $d$ and the shape regularity parameter $\sigma$ of the mesh, such that the estimate
	\[
	\iint_{T \times T'} \frac{ |I_hv(x) - I_hv(y) |^2}{|x-y|^{d+2s}} dy dx \leq \frac{C}{1-s} h_T^{d-2s} \sum_{i: \x_i \in T \cup T'} \left( \frac1{|R_i|} \int_{R_i} v(z) dz \right)^2
	\]
	holds for all $v \in L^1(\Lambda_H)$.
\end{lemma}
\begin{proof}
	In case $R_i = B_i$, a proof of the proposition above can be found in  \cite{BoNoSa18}, and the same argument is valid in case $R_i = B_i \cap \Omega$.
\end{proof}

The right hand side in Lemma \ref{lem:stabIh} may not be the most appropriate to express the stability of the operator $I_h$ because it does not involve a seminorm of $v$. To obtain an expression better suited to deal with elements contained in $\Omega$, we make two simple observations. In first place, that the quasi-interpolation operator $I_h$ preserves constant functions; secondly, that fractional-order seminorms are invariant under sums. 

\begin{lemma}[local $H^s$-stability]
	\label{lemma:bry-stab-Ih}
	Let $s \in (0,1)$ and $T, T' \in \calT_h$ with $T\subset\Omega$ and $T' \subset S^1_T$. Then, there is a constant $C$ such that the estimate
	\begin{equation}\label{eq:stability-Ih-TT'}
	\iint_{T \times T'} \frac{ |I_h v(x) - I_h v(y) |^2}{|x-y|^{d+2s}} dy dx \leq C \left[ |v|_{H^s(\widetilde{S}_T^1)}^2 + \iint_{\widetilde{S}_T^1 \times \widetilde{S}_{T'}^1} \frac{ | v(x) - v(y) |^2}{|x-y|^{d+2s}} dy dx 
	\right]
	\end{equation}
	holds for all $v \in \mathbb{V}$. Moreover, the following estimate holds:
	\begin{equation}\label{eq:bdry-stability-Ih}
	\iint_{T \times S^1_T} \frac{ |I_h v(x) - I_h v(y) |^2}{|x-y|^{d+2s}} dy dx \leq C \iint_{\widetilde{S}_T^1 \times S_T^2} \frac{ | v(x) - v(y) |^2}{|x-y|^{d+2s}} dy dx.
	\end{equation}
\end{lemma}
\begin{proof}
	Let $T$ and $T'$ be any two elements as in the hypothesis, $v \in \mathbb{V}$ and $c \in \R$ a constant to be determined. Because $\cup_{\x_i \in T \cup T'} R_i \subset \widetilde{S}_T^1 \cup \widetilde{S}_{T'}^1$ and $|R_i| \simeq h_T^d$ for every node $\x_i \in T \cup T'$, applying the Jensen's inequality we have
	\[
	\sum_{i: \x_i \in T \cup T'} \left( \frac1{|R_i|} \int_{R_i} v-c \right)^2 \le \frac{C}{h_T^d} \int_{\widetilde{S}_T^1 \cup \widetilde{S}_{T'}^1} (v-c)^2.
	\]
	Combining this bound with Lemma \ref{lem:stabIh} and the fact that $I_h (v-c) (x) - I_h (v-c)(y) = I_h v(x) - I_h v(y)$ for all $x\in T, y \in T'$, we get
	\begin{equation}\label{eq:stab-Ih-TT'}
	\iint_{T \times T'} \frac{ |I_h v(x) - I_h v(y) |^2}{|x-y|^{d+2s}} dy dx \le \frac{C}{\fran{h_T^{2s}}} \int_{\widetilde{S}_T^1 \cup \widetilde{S}_{T'}^1} (v-c)^2.
	\end{equation}
	
	We now choose $c = |\widetilde{S}_T^1|^{-1} \int_{\widetilde{S}_T^1} v$, so that we can apply the Poincar\'e inequality
	\begin{equation}\label{eq:stab-Ih-T}
	\int_{\widetilde{S}_T^1} (v-c)^2 \le C \fran{h_T^{2s}} |v|_{H^s(\widetilde{S}_T^1)}^2.
	\end{equation}
	\fran{The constant $C$ above depends on the chunkiness of $\widetilde{S}_T^1$ (see for example \cite[Proposition 1.2.6]{tesisbortha}).}
	Our choice of $c$ yields
	\[
	\int_{\widetilde{S}_{T'}^1} (v-c)^2 = \int_{\widetilde{S}_{T'}^1} \left( \frac{1}{|\widetilde{S}_{T}^1|}\int_{\widetilde{S}_{T}^1} v(x) - v(y) dy \right)^2 dx \le
	\frac{1}{|\widetilde{S}_{T}^1|}  \iint_{\widetilde{S}_{T'}^1 \times \widetilde{S}_{T}^1} |v(x) - v(y)|^2 dy dx
	\]
	and therefore, since $|x-y| \le C h_T$ for all $x\in \widetilde{S}_{T'}^1$, $y \in \widetilde{S}_{T}^1$ and $|\widetilde{S}_{T}^1| \simeq h_T^d$, we obtain
	\begin{equation}\label{eq:stab-Ih-T'}
	\int_{\widetilde{S}_{T'}^1} (v-c)^2 \le C h_T^{2s} \iint_{\widetilde{S}_T^1 \times \widetilde{S}_{T'}^1} \frac{ | v(x) - v(y) |^2}{|x-y|^{d+2s}} dy dx.
	\end{equation}.
	
	We obtain estimate \eqref{eq:stability-Ih-TT'} by combining \eqref{eq:stab-Ih-TT'}, \eqref{eq:stab-Ih-T} and \eqref{eq:stab-Ih-T'}. 	
	Summing up \eqref{eq:stability-Ih-TT'} over the elements $T' \subset S^1_T$, whose total number is less than $C_\sigma$, we immediately obtain \eqref{eq:bdry-stability-Ih}.
\end{proof}

\begin{remark}[averages]\label{rmk:promedios}
	One can readily verify that, given any two sets $A$ and $B$ and $v \in L^1(A\cup B)$, 
	\[
	\frac{1}{|A|} \int_A v(x) dx - \frac{1}{|B|} \int_B v(y) dy = \frac{1}{|A||B|} \int_A \int_B (v(x)-v(y)) dy dx .
	\]
\end{remark}

We now express the stability of $I_h$ in a way that shall be convenient to deal with elements away from one another.

\begin{lemma}[stability on non-touching elements] \label{lemma:long-range-stability}
	Let $T$ and $T'$ be any two elements such that $T \cap T' = \emptyset$. Then, for every $v \in L^2(S_T^1 \cup S_{T'}^1)$ it holds that
	\begin{equation}\label{eq:stability-Ih-TT'-far}
	\int_T \int_{T'} |I_h v(x) - I_h v(y)|^2 dy dx \le C \int_{S_T^1} \int_{S_{T'}^1} |v(x) - v(y)|^2 dy dx.
	\end{equation}
	As a consequence, given $T \in \calT_h$ it holds that
	\begin{equation*} \label{eq:long-range-stability-Ih}
	\int_T \int_{(S_T^1)^c}  \frac{|I_h v (x) - I_h v (y)|^2}{|x-y |^{d+2s}} dy dx \le C \int_{S_T^1} \int_{\R^d}   \frac{|v (x) - v (y)|^2}{|x-y |^{d+2s}} dy dx  \quad \forall v \in L^2_{loc}(\R^d).
	\end{equation*}
\end{lemma}
\begin{proof}
	Let $T,T'$ be any two disjoint elements. Thus, $\# \{ \x_i \in T \cup T' \} = 2(d+1)$, and we can consider a local node numbering such that $\x_1, \ldots, \x_{d+1} \in T$ and $\x_{d+2}, \ldots, \x_{2(d+1)} \in T'$. We write, for $x \in T$ and $y \in T'$,
	\[ \begin{aligned}
	I_h v (x) - I_h v(y) & = \sum_{i =1}^{d+1} \left( \frac1{|R_i|} \int_{R_i} v \right) \phii_i(x) - \sum_{i = d+2}^{2(d+1)} \left( \frac1{|R_i|} \int_{R_i} v \right) \phii_i(y) \\
	& = \sum_{i =1}^{d+1} \left( \frac1{|R_i|} \int_{R_i} v - \frac1{|S_{T'}^1|} \int_{S_{T'}^1} v \right) \phii_i(x) - \sum_{i = d+2}^{2(d+1)} \left( \frac1{|R_i|} \int_{R_i} v  - \frac1{|S_{T}^1|} \int_{S_{T}^1} v\right) \phii_i(y) \\
	& \quad + \frac1{|S_{T'}^1|} \int_{S_{T'}^1} v - \frac1{|S_{T}^1|} \int_{S_{T}^1} v.
	\end{aligned} \]
	
	Therefore, we can bound
	\begin{equation} \label{eq:bound-stab-Ih-far}
	|I_h v (x) - I_h v(y)|^2 \le 3 (A_1^2 + A_2^2 + A_3^2),
	\end{equation}
	with
	\[ \begin{aligned}
	& A_1 =  \sum_{i =1}^{d+1} \left( \frac1{|R_i|} \int_{R_i} v - \frac1{|S_{T'}^1|} \int_{S_{T'}^1} v \right) \phii_i(x), \quad  A_2 =  \sum_{i = d+2}^{2(d+1)} \left( \frac1{|R_i|} \int_{R_i} v  - \frac1{|S_{T}^1|} \int_{S_{T}^1} v\right) \phii_i(y) , \\
	& A_3 = \frac1{|S_{T'}^1|} \int_{S_{T'}^1} v - \frac1{|S_{T}^1|} \int_{S_{T}^1} v.
	\end{aligned} \]
	
	Because $|\phii_i|\le 1$, $|R_i| \simeq h_T^d \simeq |S_T^1|$ for all $i = 1, \ldots, d+1$ and $\cup_{i=1}^{d+1} R_i \subset S_T^1$, and by using Remark \ref{rmk:promedios} and the Jensen's inequality, we can bound
	\[
	A_1^2 \le C \sum_{i =1}^{d+1} \left( \frac1{|R_i|} \int_{R_i} v - \frac1{|S_{T'}^1|} \int_{S_{T'}^1} v \right)^2 \le \frac{C}{|S_{T}^1||S_{T'}^1|} \int_{S_T^1} \int_{S_{T'}^1} |v(t) - v(w)|^2 dt dw.
	\]
	
	In the same fashion, one readily obtains
	\[
	A_2^2, A_3^2 \le \frac{C}{|S_{T}^1||S_{T'}^1|} \int_{S_T^1} \int_{S_{T'}^1} |v(w) - v(t)|^2 dt dw,
	\]
	and collecting the bounds for the $A_j$'s and integrating \eqref{eq:bound-stab-Ih-far} over $T\times T'$, we readily obtain \eqref{eq:stability-Ih-TT'-far}.
	
	Naturally, $T \cap T' = \emptyset$ is equivalent to $T' \in (S_T^1)^c$ \fran{or $d(T, T') > 0$. Thus, we have
	\[
	|x-y| \ge d(T, T') > 0, \quad |t-w|\le Cd(T, T') \quad \forall x \in T, \ y \in T', \ t \in S_T^1, \ w \in S_{T'}^1,
	\]
and we can use \eqref{eq:stability-Ih-TT'-far} to write
	\[ \begin{aligned}
	\int_T \int_{T'}  \frac{|I_h v (x) - I_h v (y)|^2}{|x-y |^{d+2s}} dy dx & \le d(T, T')^{-(d+2s)} \int_T \int_{T'}  |I_h v (x) - I_h v (y)|^2  dy dx 	\\
& \le	C d(T, T')^{-(d+2s)}
	\int_{S_T^1} \int_{S_{T'}^1} |v (w) - v (t)|^2 dtdw \\ 
	&\le C 
	 \int_{S_T^1} \int_{S_{T'}^1} \frac{|v (w) - v (t)|^2}{|w-t|^{d+2s}} dtdw.
	\end{aligned} \]	
	}
Estimate \eqref{eq:stability-Ih-TT'-far} follows by summing on elements $T' \subset (S_T^1)^c \cap \Lambda_H$ and recalling that $I_h$ vanishes on $\Lambda_H^c$. 
\end{proof}

We shall also require the following auxiliary result, that is proved by means of the same kind of arguments as in \cite[Proposition 3.4]{BoLiNo19analysis}

\begin{lemma}[local $L^2$ interpolation error] \label{lemma:conv-L2c}
	Assume $v \in L_{loc}^\infty(\Rd)$. Then, if the computational domains $\{ \Lambda_H \}$ are taken according to \eqref{eq:defofLambda} with $H \to \infty$ as $h \to 0$, we have
	\[
	\| v- I_h v \|_{L^2_{loc}(\Rd)} \to 0, \quad \mbox{as } h \to 0.
	\]
\end{lemma}

\begin{proof}
	Let $K \subset \R^d$ be a bounded set and $x \in K$. Then, there exists $h_0$ sufficiently small such that $K \subset \Lambda_H$ for all $h < h_0$. Thus, we may assume that $x\in T$ for some $T \in \calT_h$. Furthermore, let us assume that $x$ is a Lebesgue point of $v$. Then, we have
	\[
	v(x) - I_h v (x) =  v(x) -  \sum_{i \colon \x_i \in T} \left(\frac1{|R_i|} \int_{R_i} v(y) dy \right) \phii_i (x) = \sum_{i \colon \x_i \in T} \left(\frac1{|R_i|} \int_{R_i} (v(x) -  v(y)) dy \right) \phii_i (x).
	\]
	We exploit that for all $i$\fran{,} $|\phii_i| \le 1$, $|R_i| \simeq |T| \simeq h_T^d$ by shape regularity. Also,  the definition of the region $R_i$ gives $R_i \subset S_i \subset S^1_T$ and, in turn, we have $S^1_T \subset B_{r}(x)$ with a radius $r = C h_T$. We get
	\[ \begin{aligned}
	|v(x) - I_h v (x)| & \le \frac{C}{h_T^d} \int_{\cup_i R_i} |v(x) -  v(y)| dy \le \frac{C}{h_T^d} \int_{B_{Ch_T}(x)} |v(x) -  v(y)| dy  \\
	& \le \frac{C}{|B_{Ch_T}(x)|} \int_{B_{Ch_T}(x)} |v(x) -  v(y)| dy \to 0 \quad \mbox{as } h \to 0,
	\end{aligned}\]
	because $x$ is a Lebesgue point of $v$. Therefore, by the Lebesgue Differentiation Theorem we deduce that $I_h v \to v$ a.e. \fran{in $K$}.
	
	Moreover, because $v \in L_{loc}^\infty(\Rd)$ we have $|I_h v| \le \| v \|_{L^\infty(K)}$ and since $|K|$ is finite we apply the Dominated Convergence Theorem to conclude that
	\[
	\lim_{h \to 0} \int_K |I_h v (x) - v(x)|^2 dx = 0.
	\]
	This finishes the proof.
\end{proof}

Finally, we have some estimates at infinity.

\begin{lemma}[tail of interpolation error] \label{lem:tail-interpolation}
	Let $R > 0$ be sufficiently large. Then, if $H > R$ and $h \le 1$, we have
	\[
	\int_\Omega \int_{\Lambda_R^c} \frac{|(v-I_h v)(x)-(v-I_h v)(y)|^2}{|x-y|^{d+2s}} \, dy dx \le C \left(\frac{\| v-I_h v \|_{L^2(\Omega)}^2}{R^{2s}} +  \int_{\Lambda_{R-1}^c} \frac{|v(y)|^2}{|y|^{d+2s}} \, dy \right)
	\]
\end{lemma}
\begin{proof}
	We split
	\[ \begin{split}
	\int_\Omega \int_{\Lambda_R^c} \frac{|(v-I_h v)(x)-(v-I_h v)(y)|^2}{|x-y|^{d+2s}} \, dy dx & \le  2 \int_\Omega \int_{\Lambda_R^c} \frac{|(v-I_h v)(x)|^2}{|x-y|^{d+2s}} \, dy dx \\ 
	& + 2 \int_\Omega \int_{\Lambda_R^c} \frac{|v-I_h v)(y)|^2}{|x-y|^{d+2s}} \, dy dx.
	\end{split} \]
	Using \eqref{eq:decay-Omega-Lambda}, the first integral in the right hand side can be bounded by
	\[
	\int_\Omega \int_{\Lambda_R^c} \frac{|(v-I_h v)(x)|^2}{|x-y|^{d+2s}} \, dy dx \le \frac{C}{R^{2s}} \| v-I_h v \|_{L^2(\Omega)}^2.
	\]
	
	As for the second one, we now use \eqref{eq:decay-Lambda-Omega} to obtain
	\[
	\int_\Omega \int_{\Lambda_R^c} \frac{|(v-I_h v)(y)|^2}{|x-y|^{d+2s}} \, dy dx \le C \int_{\Lambda_R^c} \frac{|(v-I_h v)(y)|^2}{|y|^{d+2s}} \, dy \le C \left(\int_{\Lambda_R^c} \frac{|v(y)|^2}{|y|^{d+2s}} \, dy + \int_{\Lambda_R^c} \frac{|I_h v(y)|^2}{|y|^{d+2s}} \, dy \right),
	\]
	and because $I_h v$ vanishes on $\Lambda_H^c$, we have
	\[
	\int_{\Lambda_R^c} \frac{|I_h v(y)|^2}{|y|^{d+2s}} \, dy = \int_{\Lambda_H\setminus \Lambda_R} \frac{|I_h v(y)|^2}{|y|^{d+2s}} \, dy.
	\]
	Take any element $T \in \calT_h$ such that $T \cap (\Lambda_H\setminus \Lambda_R) \neq \emptyset$. For $y \in T$ we thus have
	\[
	|I_h v (y)|^2 \le C \sum_{i \colon \x_i \in T} \left(\frac1{|B_i|} \int_{B_i} v^2 \right),
	\]
	and because \fran{$|B_i| \simeq h^d_T$} and $|z| \simeq |y|$ for all $y \in T$, $z \in S_T^1$, we can write
	\[
	\int_{\fran{T\cap(\Lambda_H\setminus \Lambda_R)}} \frac{|I_h v(y)|^2}{|y|^{d+2s}} \, dy \le C \int_{S_T^1} \frac{|v(z)|^2}{|z|^{d+2s}} dz.
	\]
	Summing up in all the elements and using that $h \le 1$, we conclude that
	\[
	\int_{\Lambda_R^c} \frac{|I_h v(y)|^2}{|y|^{d+2s}} \, dy \le C \int_{\Lambda_{R-1}^c} \frac{|v(y)|^2}{|y|^{d+2s}} \, dy.
	\]
\end{proof}

\subsection{Convergence}

We next prove the convergence of the finite element approximations by combining the various interpolation estimates derived in last section with the regularity of solutions. We require solutions to be locally bounded.

\begin{theorem}[convergence] \label{thm:convergence}
	Let $s \in (0,1)$, $\alpha > 0$, $f \in L^2(\Omega)$, 
\fran{$g \in [L^2_{s} (\Omega^c)]'$}	for some $\beta \in (0,s)$ and $u$ be the solution to \eqref{eq:weak_form}. Let $u_h$ be the finite element solution computed on a mesh with size $h = \max_{T \in \calT_h} h_T$. Then, \fran{if the computational domains $\{ \Lambda_H \}$ are taken according to \eqref{eq:defofLambda} with $H \to \infty$ as $h \to 0$ and}	
	assuming $u \in L_{loc}^\infty(\overline{\Omega^c})$, it holds that
	\[
	\lim_{h \to 0} \| u - u_h \|_\V = 0 .
	\]
\end{theorem}
\begin{proof} Because of the best approximation property \eqref{eq:best-approximation}, it suffices to estimate the interpolation error. Clearly, using \eqref{eq:L2-interpolation-error}, we immediately deduce that the $L^2$-interpolation error over $\Omega$ tends to zero. Namely,
	\begin{equation} \label{eq:interpolation-error-Omega}
	\begin{split}
	\| u - I_h u \|_{L^2(\Omega)}^2 & = \sum_{T \subset \overline\Omega}\| u - I_h u \|_{L^2(T)}^2 \\ 
	& \le C  \sum_{T \subset \overline\Omega} h_T^{2s} |u|_{H^s(S_T^1 \cap \Omega)}^2 \le C h^{2s}|u|_{H^s(\Omega)}^2 \to 0, \quad \mbox{as } h \to 0,
	\end{split}	\end{equation}
	where we recall that the family $\{ \calT_h \}$ is assumed to mesh $\overline \Omega$ exactly.
	
	In order to estimate the interpolation error in the $\X$-seminorm, we let $\eps > 0$ be any positive number. Because $u \in \V$, there exist $\delta > 0$ and $R >0$ such that
	\begin{equation} \label{eq:chico-u} \begin{split}
	& \int_{\Omega\setminus\Omega_{2\delta}} \int_{\R^d} \frac{|u(x)-u(y)|^2}{|x-y|^{d+2s}} dy dx  < \eps, \\
	& \int_{\Omega} \int_{\Lambda_{R-1}^c} \frac{|u(x)-u(y)|^2}{|x-y|^{d+2s}} dy dx  < \eps,
	\end{split}	\end{equation}
	where we introduced the notation 
	\[
	\Omega_r = \{ x \in \Omega \colon d(x,\pp\Omega) \ge r \}, \quad r > 0 .
	\]
	For convenience, we shall denote $\calT_h^r = \{ T \in \calT_h \colon T \cap \Omega_r \neq \emptyset \}$ and, without loss of generality, assume  that $h \le \delta /8 \le 1$ and $H > R$. We decompose the $\X$-seminorm as
	\begin{equation} \label{eq:splitting-X}
	|u - I_h u|_\X^2 \le 2 \int_{\Omega} \int_{\R^d} \frac{|(u - I_h u) (x) - (u - I_h u) (y)|^2}{|x-y |^{d+2s}} \; dy \; dx = 2 I_1 + 2 I_2 +  2 I_3,
	\end{equation}
	where
	\[ \begin{split}
	&	I_1 = \int_{\Omega_\delta} \int_{\Lambda_R} \frac{|(u - I_h u) (x) - (u - I_h u) (y)|^2}{|x-y |^{d+2s}} \; dy \; dx, \\
	&	I_2 = \int_{\Omega_\delta} \int_{\Lambda_R^c} \frac{|(u - I_h u) (x) - (u - I_h u) (y)|^2}{|x-y |^{d+2s}} \; dy \; dx,\\
	&	I_3 = \int_{\Omega \setminus \Omega_\delta} \int_{\Rd} \frac{|(u - I_h u) (x) - (u - I_h u) (y)|^2}{|x-y |^{d+2s}} \; dy \; dx.
	\end{split} \]

	Let us first consider the term $I_1$ above, that can be bounded as
	\begin{equation} \label{eq:split-I1} \begin{aligned}
	I_1 \le & \sum_{T \in \calT_h^\delta} \int_T \int_{S^1_T} \frac{|(u - I_h u) (x) - (u - I_h u) (y)|^2}{|x-y |^{d+2s}} \; dy \; dx \\ & + \sum_{T \in \calT_h^\delta} \int_T \int_{(S^1_T)^c \cap \Lambda_R} \frac{|(u - I_h u) (x) - (u - I_h u) (y)|^2}{|x-y |^{d+2s}} \; dy \; dx .
	\end{aligned} \end{equation}
	By Theorem \ref{thm:interior_regularity}, we have $u \in \cap_{\sigma > 0} H^{2s-\sigma}_{loc}(\Omega)$. Therefore, fixing some $\sigma \in (0,s)$ and applying \eqref{eq:Hs-interpolation-error}, we obtain
	\begin{equation} \label{eq:split-I1-patch} \begin{aligned}
	\sum_{T \in \calT_h^\delta} \int_T \int_{S^1_T} \frac{|(u - I_h u) (x) - (u - I_h u) (y)|^2}{|x-y |^{d+2s}} \; dy \; dx & \le C \sum_{T \in \calT_h^\delta} h_T^{2(s-\sigma)} | u |_{H^{2s-\sigma}(S^2_T)}^2 \\
	&  \le C h^{2(s-\sigma)} | u |_{H^{2s-\sigma}(\Omega_{\delta/2})}^2 
	\to 0 , \quad \mbox{as } h \to 0.
	\end{aligned} \end{equation}
	
	To deal with the second sum in \eqref{eq:split-I1}, we split it as
	\begin{equation} \begin{aligned} \label{eq:split-far-I1}
	\sum_{T \in\calT_h^\delta} & \int_T \int_{(S^1_T)^c \cap \Lambda_R} \frac{|(u - I_h u) (x) - (u - I_h u) (y)|^2}{|x-y |^{d+2s}} \; dy \; dx \\ 
	& \le 2 \sum_{T \in \calT_h^\delta} \int_T \int_{(S^1_T)^c} \frac{|(u - I_h u) (x)|^2}{|x-y |^{d+2s}} \; dy \; dx + 2 \sum_{T \in \calT_h^\delta} \int_T \int_{(S^1_T)^c \cap \Lambda_R} \frac{|(u - I_h u) (y)|^2}{|x-y |^{d+2s}} \; dy \; dx ,
	\end{aligned} \end{equation}
	and also  remark that, for every $T \in \calT_h$ and $x \in T$,
	\begin{equation} \label{eq:polares-patch}
	\int_{(S^1_T)^c} \frac{1}{|x-y |^{d+2s}} \; dy \le \frac{C}{h_T^{2s}}.
	\end{equation}
	
	For the first sum in the right hand side in \eqref{eq:split-far-I1}, we exploit \eqref{eq:polares-patch}, apply \eqref{eq:L2-interpolation-error} and use the interior $H^{2s-\sigma}$-regularity of $u$ from Theorem \ref{thm:interior_regularity} to deduce
	\[ \begin{aligned}
	\sum_{T \in  \calT_h^\delta} \int_T \int_{(S^1_T)^c} \frac{|(u - I_h u) (x)|^2}{|x-y |^{d+2s}} \; dy \; dx & \le C \sum_{T \in \calT_h^\delta}  \frac{ \| u - I_h u \|^2_{L^2(T)}}{h_T^{2s}}
	\\ & \le C \sum_{T \in\calT_h^\delta} h_T^{2(s-\sigma)} | u |_{H^{2s-\sigma}(S^2_T)}^2 \to 0 , \quad \mbox{as } h \to 0.
	\end{aligned} \]
	
	We can deal with the last sum in \eqref{eq:split-far-I1} by using Lemma \ref{lemma:symmetry}. Indeed, by applying it and using \eqref{eq:polares-patch}, we get
	\begin{equation} \label{eq:split-delta-y}\begin{aligned}
	\sum_{T \in \calT_h^\delta} \int_T \int_{(S^1_T)^c \cap \Lambda_R} \frac{|(u - I_h u) (y)|^2}{|x-y |^{d+2s}} \; dy \; dx & \le 
	\sum_{T \in \calT_h \colon T \cap \Lambda_R \neq \emptyset} \int_T \int_{(S^1_T)^c \cap \Lambda_R} \frac{|(u - I_h u) (y)|^2 \chi_{\Omega_{3\delta/4}}(x)}{|x-y |^{d+2s}} \; dy \; dx \\
	& = \sum_{T \in \calT_h \colon T \cap \Lambda_R \neq \emptyset} \int_T \int_{(S^1_T)^c \cap \Lambda_R} \frac{|(u - I_h u) (x)|^2 \chi_{\Omega_{3\delta/4}}(y)}{|x-y |^{d+2s}} \; dy \; dx \\
	& \le  C \sum_{T \in \calT_h \colon T \cap \Lambda_R \neq \emptyset} \frac{\| u - I_h u \|_{L^2(T)}^2}{d(T, (S^1_T)^c\cap\Omega_{3\delta/4})^{2s}}. 
	\end{aligned} \end{equation}
	
	We now distinguish three cases in the last sum above. For the elements contained in $\Omega^c$, we use Lemma \ref{lemma:conv-L2c} and the fact that $d(T, (S^1_T)^c\cap\Omega_{3\delta/4}) \ge \delta/4$ if $T \subset \Omega^c$, to deduce that
	\begin{equation} \label{eq:split-delta-y-c}
	\sum_{T \in \calT_h \colon T \subset \Omega^c, T \cap \Lambda_R \neq \emptyset} \frac{\| u - I_h u \|_{L^2(T)}^2}{d(T, (S^1_T)^c\cap\Omega_{3\delta/4})^{2s}} \le C \frac{\| u - I_h u \|_{L^2(\Lambda_{R+1})}^2}{\delta^{2s}} \to 0, \quad \mbox{as } h \to 0.
	\end{equation}
	
	The elements in $\calT_h^\delta$ can be treated by using Theorem \ref{thm:interior_regularity} and \eqref{eq:L2-interpolation-error},
	\begin{equation} \label{eq:split-delta-y-delta}
	\sum_{T \in \calT_h^\delta} \frac{\| u - I_h u \|_{L^2(T)}^2}{d(T, (S^1_T)^c\cap\Omega_{3\delta/4})^{2s}} \le C \sum_{T \in \calT_h^\delta} \frac{\| u - I_h u \|_{L^2(T)}^2}{h_T^{2s}} \le C \sum_{T \in \calT_h^\delta} h_T^{2(s-\sigma)} | u |_{H^{2s-\sigma}(S^1_T)}^2 \to 0,
	\end{equation}
	as $h \to 0.$
	
	For those elements contained in $\Omega$ but not belonging to $\calT_h^\delta$, we also use \eqref{eq:L2-interpolation-error}, but now we critically exploit the choice of $\delta$ in \eqref{eq:chico-u},
	\begin{equation} \label{eq:split-delta-y-bdry} \begin{aligned}
	\sum_{T \in \calT_h \setminus \calT_h^\delta \colon T \subset \overline\Omega} \frac{\| u - I_h u \|_{L^2(T)}^2}{d(T, (S^1_T)^c\cap\Omega_{3\delta/4})^{2s}} & \le C \sum_{T \in \calT_h \setminus \calT_h^\delta \colon T \subset \overline\Omega} \frac{\| u - I_h u \|_{L^2(T)}^2}{h_T^{2s}} \\ 
	& \le C \sum_{T \in \calT_h \setminus \calT_h^\delta \colon T \subset \overline\Omega} | u |_{H^s(S_T^1 \cap \Omega)}^2 \\
	& \le C |u |_{H^s(\Omega \setminus \Omega_{2\delta})}^2 < C \eps.
	\end{aligned} \end{equation}
	
	Substituting \eqref{eq:split-delta-y-c}, \eqref{eq:split-delta-y-delta} and \eqref{eq:split-delta-y-bdry} in \eqref{eq:split-delta-y}, we deduce that
	\[
	\sum_{T \in \calT_h^\delta} \int_T \int_{(S^1_T)^c \cap \Lambda_R} \frac{|(u - I_h u) (y)|^2}{|x-y |^{d+2s}} \; dy \; dx \le
	C \eps + \mathcal{O}(1),
	\]
	and in turn, combining this estimate with \eqref{eq:split-I1-patch} and going back to \eqref{eq:split-I1}, we obtain
	\begin{equation} \label{eq:bound-I1}
	I_1 \le C \eps + \mathcal{O}(1).
	\end{equation}
	
	Next, we analyze the term $I_2$ in \eqref{eq:splitting-X}, which involves interactions between $\Omega_\delta$ and the unbounded set $\Lambda_R^c$. For that purpose, we combine Lemma \ref{lem:tail-interpolation} with \eqref{eq:interpolation-error-Omega} and \eqref{eq:chico-u}
	\begin{equation} \label{eq:bound-I2}
	I_2 \le C \left(\frac{\| u-I_h u \|_{L^2(\Omega)}^2}{R^{2s}} +  \int_{\Lambda_{R-1}^c} \frac{|u(y)|^2}{|y|^{d+2s}} \, dy \right) \le \mathcal{O}(1) +  C \eps.
	\end{equation}
	
	Let us finally consider the term $I_3$ in \eqref{eq:splitting-X}, which accounts for interactions between $\Omega \setminus \Omega_\delta$ --a boundary layer of width $\delta$ in $\Omega$-- and $\Rd$. Our argument needs to be of a different nature to the one that we performed for $I_1$ and $I_2$: now we cannot exploit interior regularity. Nevertheless, $I_3$ is expected to be small because it involves integration over a region whose contribution to the $\X$-seminorm of $u$ is roughly $\eps$ (cf. \eqref{eq:chico-u}). Thus, to deal with $I_3$ it suffices to exploit local stability properties of the interpolation operator $I_h$.
	
	Accordingly, we split $I_3$ as the sum of two integrals, one involving $u$ and another involving $I_h u$:
	\begin{equation} \begin{aligned} \label{eq:split-I3}
	I_3 & = \int_{\Omega \setminus \Omega_\delta} \int_{\R^d} \frac{|(u - I_h u) (x) - (u - I_h u) (y)|^2}{|x-y |^{d+2s}} \; dy \; dx \\
	& \le 2 \int_{\Omega \setminus \Omega_\delta} \int_{\R^d} \frac{|u (x) - u (y)|^2}{|x-y |^{d+2s}} \; dy \; dx
	+ 2 \int_{\Omega \setminus \Omega_\delta} \int_{\R^d} \frac{|I_h u (x) - I_h u (y)|^2}{|x-y |^{d+2s}} \; dy \; dx \\
	& \le 2 \eps + 2 \int_{\Omega \setminus \Omega_\delta} \int_{\R^d} \frac{|I_h u (x) - I_h u (y)|^2}{|x-y |^{d+2s}} \; dy \; dx.
	\end{aligned} \end{equation}
	
	We need to bound the last integral in the right hand side above. For that purpose, we observe that $\Omega \setminus \Omega_\delta \subset \{ T \in \calT_h \setminus \calT_h^{3\delta/2} \colon T \subset \Omega \}$ and decompose
	\[ \begin{aligned}
	\int_{\Omega \setminus \Omega_\delta} \int_{\R^d} \frac{|I_h u (x) - I_h u (y)|^2}{|x-y |^{d+2s}} \; dy \; dx \le
	& \sum_{T \in \calT_h \setminus \calT_h^{3\delta/2} \colon T \subset \overline\Omega}  \int_T \int_{S^1_T} \frac{|I_h u (x) - I_h u (y)|^2}{|x-y |^{d+2s}} \; dy \; dx \\
	& +  \sum_{T \in \calT_h  \setminus \calT_h^{3\delta/2} \colon T \subset \overline\Omega} \int_T \int_{(S^1_T)^c} \frac{|I_h u (x) - I_h u (y)|^2}{|x-y |^{d+2s}} \; dy \; dx.
	\end{aligned} \]
	
	We exploit Lemma \ref{lemma:bry-stab-Ih} and the assumption $h \le \delta/8$ to treat the first sum:
	\begin{equation} \label{eq:split-I3-patch} \begin{aligned}
	\sum_{T \in \calT_h  \setminus \calT_h^{3\delta/2} \colon T \subset \overline\Omega} & \int_T \int_{S^1_T} \frac{|I_h u (x) - I_h u (y)|^2}{|x-y |^{d+2s}} \; dy \; dx 
	& \le C \int_{\Omega \setminus \Omega_{2\delta}} \int_{\R^d} \frac{|u (x) - u (y)|^2}{|x-y |^{d+2s}} \; dy \; dx < C \eps .
	\end{aligned} \end{equation}
	
	Next, we apply Lemma \ref{lemma:long-range-stability} to deduce that
	\begin{equation} \label{eq:split-I3-away} \begin{aligned}
	\sum_{T \in \calT_h  \setminus \calT_h^{3\delta/2} \colon T \subset \overline\Omega}
	\int_T \int_{(S^1_T)^c} \frac{|I_h u (x) - I_h u (y)|^2}{|x-y |^{d+2s}} \; dy \; dx & \le C \sum_{T \in \calT_h  \setminus \calT_h^{3\delta/2} \colon T \subset \overline\Omega}
	\int_{S_T^1} \int_{\R^d} \frac{|u (x) - u (y)|^2}{|x-y |^{d+2s}} \; dy \; dx \\
	& \le C \int_{\Omega \setminus \Omega_{2\delta}} \int_{\R^d} \frac{|u (x) - u (y)|^2}{|x-y |^{d+2s}} \; dy \; dx < C \eps .
	\end{aligned} \end{equation}
	
	Substituting \eqref{eq:split-I3-patch} and \eqref{eq:split-I3-away} in \eqref{eq:split-I3}, we obtain
	\begin{equation} \label{eq:bound-I3}
	I_3 \le C \eps.
	\end{equation}
	
	Finally, collecting \eqref{eq:bound-I1}, \eqref{eq:bound-I2}, \eqref{eq:bound-I3} and \eqref{eq:splitting-X}, we conclude that
	\[
	|u - I_h u|_\X^2 \le C \eps + \mathcal{O}(1).
	\]
	The result follows because $\eps > 0$ is arbitrary.
\end{proof}

\fran{\begin{remark}[convergence rates under regularity assumptions]
If, besides the hypotheses from Theorem \ref{thm:convergence}, we assume that the solution $u$ belongs to $H^r_{loc}(\Rd)$ for some $r \in (s,2]$, then it is clear (cf. \eqref{eq:Hs-interpolation-error}) that the local interpolation error is of the order of $h^{r-s}$. Furthermore, if $u \in L^2_{s-\beta}(\R^d)$ for some $\beta \in (0,s)$ --which is guaranteed by Proposition \ref{prop:further_decay} as long as $g \in [L^2_{s+\beta}(\Omega^c)]'$-- then using \eqref{eq:decay_measure} we have
\[
\int_\Omega \int_{\Lambda_H^c} \frac{|u(y) - I_h u(y)|^2}{|x-y|^{d+2s}} \, dy dx \le C H^{-2\beta} \| u \|_{L^2_{s-\beta}(\R^d)}^2.
\]

Therefore, if we take the computational domain diameter so that it satisfies $H^{-\beta} \le C h^{r-s}$, namely $H \ge C h^{\frac{s-r}\beta}$, a direct calculation shows that we have convergence with order $h^{r-s}$ with respect to the mesh size:
\[
\| u - u_h \|_{\V} \le C h^{r-s}.
\]
\end{remark}
}

\section{Numerical experiments} \label{sec:experiments}
In this section we perform numerical experiments that illustrate the convergence of the finite element discretizations and the effect of truncating the computational domain. We also present an example in a two-dimensional setting in which the value of $s$ dictates the behavior of solutions at infinity. As an application of our finite element scheme, we discretize the heat equation for the fractional Laplacian and display the convergence  as $t \to \infty$ of the discrete solution towards the mean value of the initial condition.

\subsection{Explicit non-trivial solutions} \label{sec:solutions}

As we discussed in Remark \ref{rmk:sol-constante}, a trivial explicit solution of \eqref{eq:Neumann} can be obtained by taking $f \equiv \alpha$ and $g \equiv 0$. In such a case, the solution $u \equiv 1$ is approximated in an exact form by our numerical scheme. In order to test our method, we construct some non-trivial solutions as follows: assume that $w: \Rd \to \R$ is a solution of the nonhomogeneous Dirichlet problem
\begin{equation} \label{eq:Dirichlet}
\left\lbrace \begin{array}{rl}
\Laps w  = f_D & \mbox{in } \Omega, \\
w = h & \mbox{in } \Omega^c,
\end{array} \right.
\end{equation}
where $f_D$ and $h$ are some known functions. Then,  defining
\begin{equation*}
g(x) :=  h(x) \, C_{d,s}  \int_\Omega \frac{1}{|x-y|^{d+2s}} \, dy - C_{d,s} \int_\Omega \frac{w(y)}{|x-y|^{d+2s}} \, dy,
\end{equation*}      
for all $x \in \Omega^c$, and using relation  \eqref{eq:dato_afuera}, it follows that $w$ also solves
\begin{equation*} \label{eq:sol_expl}
\left\lbrace \begin{array}{rl}
\Laps u + \alpha u = f_D + \alpha w & \mbox{in } \Omega, \\
\Ns u = g & \mbox{in } \Omega^c.
\end{array} \right. 
\end{equation*}
Thus, we can construct explicit examples by building from known solutions of \eqref{eq:Dirichlet} for which the computation of $g$ can be numerically handled.     

\subsection{Convergence order} \label{sec:ej_1d}
Following the former ideas, we consider $\Omega = (-1,1)$, $\alpha = 1$, and
\begin{equation} 
w(x) = \left\lbrace \begin{array}{rl}
c_{s}(1-x^2)^s & \mbox{in } \Omega, \\
0 & \mbox{in } \Omega^c,
\end{array} \right.
\label{eq:sol_exp}
\end{equation}
with 
\[
c_s = \frac{ \sqrt{\pi} }{ 2^{2s}  \Gamma( \frac{1 + 2s}{2} ) \Gamma(1+s) }.
\]

This function $w$ is a well-known solution of \eqref{eq:Dirichlet} with $f_D \equiv 1$ and $h \equiv 0$. We thereby set $f = 1 + w$, and $g(x) = - C_{1,s}\int^1_{-1} \frac{w(y)}{|x-y|^{1+2s}} \, dy$ in the Neumann problem \eqref{eq:Neumann}. Note that here $C_{1,s}$ is the constant defined in \eqref{eq:defofLaps}. 

We point out that in this case the function $g$ has a singularity on $-1$ and $1$. More precisely, for $\delta > 0$, both $g(1 + \delta)$ and $g(-1-\delta)$ are of order $\mathcal{O}(\delta^{-s})$ near the interval endpoints (see \cite[Remark 5.2.5]{tesisbortha} for details). Thus, the nonlocal flux density satisfies $g \in L^2( \Omega^c )$ only when $s<1/2$ and, in this example, two numerical challenges arise in the assembly of the right hand side. Namely, the computation of $\int_{\Omega^c} g \phii_i$ when $\supp (\phii_i) \subset \Omega^c$ with $\supp( \phii_i ) \cap \partial \Omega \not = \emptyset$, and the computation of $\int_{\Omega^c} g \phii_{N+1}$, where $\phii_{N+1}$ is the constant basis function over $\Lambda^c_H$. In the first case we have to deal with a singular integrand, while in the second one we need to compute an integral over an unbounded domain.

Since $g(x) \simeq |x|^{-1-2s}$ for large values of $|x|$, the integral

$$\int_{\Omega^c} g \phii_{N+1} = \int^{+\infty}_{H} g(x) \, dx + \int^{-H}_{-\infty} g(x) \, dx = 2\int^{+\infty}_{H} g(x) \, dx,$$   
can be approximated by means of standard techniques. On the other hand, we deal with the first difficulty by a careful treatment of the singularity in order to avoid numerical issues. This is detailed in Appendix \ref{app:g_construida}.  

We display convergence orders for several values of $s$ in Figure \ref{fig:ordenes}. Because $g \not \in L^2(\Omega^c)$ for $s \geq 1/2$, we restrict ourselves to the range $s \leq 1/2$. Although we emphasize that the condition $H \to \infty$ as $h \to 0$ is needed in general, in these experiments the choice of $H = \mbox{diam}(\Lambda_H)$ does not seem to affect the convergence rate. This is possibly due to the fact that the solution $w$ is constant in $\Omega^c$ and therefore it can be exactly represented by the basis function $\phii_{N+1}$ on $\Lambda_H^c$.

\begin{figure}[h]
	\centering
	\begin{tabular}{cc}
		\subf{\includegraphics[width=50mm]{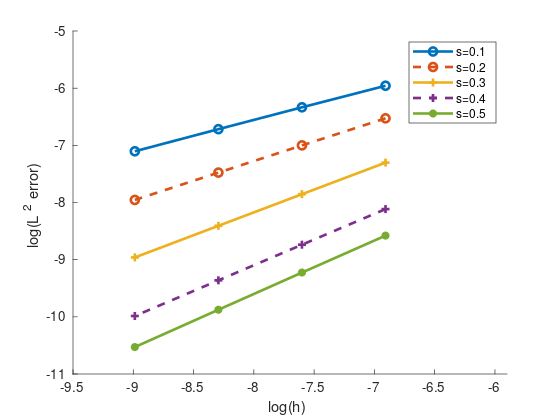}}
		{$\| \cdot \|_{L^2(\Omega)}$}
		&
		\subf{\includegraphics[width=50mm]{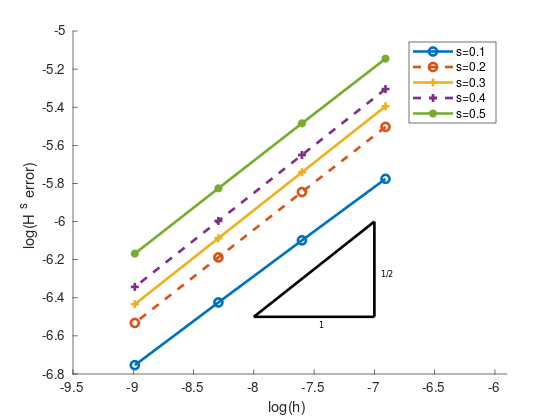}}
		{ $\| \cdot \|_{H^s(\Omega)}$}
		\\
	\end{tabular}
	\caption{The $L^2(\Omega)$ and $H^s(\Omega)$ errors in logarithmic scale for Example \ref{sec:ej_1d}, using several values of $s$. In these experiments we used  uniform meshes with $h = 1/1000,$ $1/2000,$ $1/4000,$ $1/8000,$ and $H = 1.2$. The observed order of convergence is approximately $s + 1/2$ and $1/2$ in the $L^2(\Omega)$ and $H^s(\Omega)$ norms, respectively. }
	\label{fig:ordenes}
\end{figure}

\subsection{Convergence in $H$} \label{sec:ej_R}
In this example we consider $\Omega = (-1,1)$, $f \equiv 1$, and $g(x) = -1/|x|^{1+p}$ for some $p>0$, and we aim to find experimental convergence rates in $H = \mbox{diam}(\Lambda_H)$, using a fixed uniform mesh with small $h$. We shall denote by $u^H_h$ the discrete solution computed on a mesh with size $h$ and a computational domain $\Lambda_H$. We are interested in the behavior of $\|u^{H_n}_h -u^{H_{n+1}}_h \|_{L^2(\Omega)}$, with $\{H_n\} \subset \R_+$ and $H_{n+1} - H_{n} \simeq k$ for some fixed constant $k>0$. Numerical results for $s = 0.3$, $s = 0.8$, and several choices of $g$ are shown in Figure \ref{fig:ordenes_R}. These experiments suggest that $\|u^{H_n}_h -u^{H_{n+1}}_h \|_{L^2(\Omega)} \lesssim H^{-c}$ for some $c>0$ depending on both $s$ and $g$. Table \ref{tab:ordenes} displays least-square fittings of the exponent $c$.       

\begin{figure}[h]
	\centering
	\begin{tabular}{cc}
		\subf{\includegraphics[width=50mm]{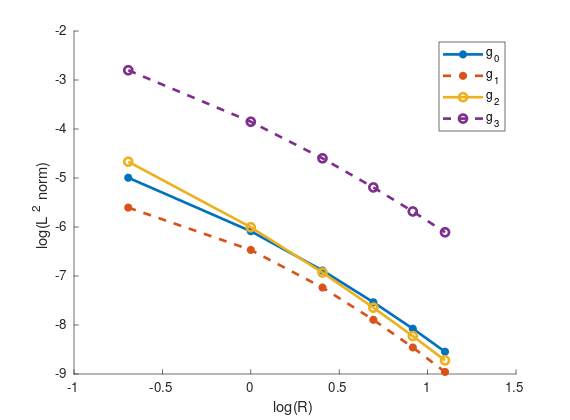}}
		{$s = 0.3$}
		&
		\subf{\includegraphics[width=50mm]{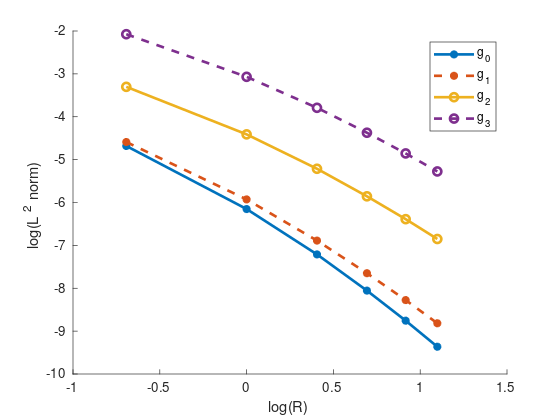}}
		{ $s = 0.8$}
		\\
	\end{tabular}
	\caption{Experimental results for Example \ref{sec:ej_R}. We plot $\log\big( \|u^{H_n}_h -u^{H_{n+1}}_h \|_{L^2(\Omega)} \big)$ vs $\log( H_{n+1} )$ for $s=0.3$ and $s=0.8$. In these experiments, we set $h = 1/1000$, $H =\{$ 0.1, 0.5, 1, 1.5, 2, 2.5, 3, 3.5$\}$, and the right hand sides $g_0 \equiv 0$, $g_1(x) = -|x|^{-2}$, $g_2(x) = -|x|^{-1.5}$, and $g_3(x) = -|x|^{-1.2}$. Also, we used $f \equiv 1$ for all cases except for $g_0$, where we took $f(x) = \sin(\pi x)$ in order to avoid  trivial constant solutions. }  
	\label{fig:ordenes_R}
\end{figure}

\begin{table}[]
	\begin{tabular}{|c|c|c|c|c|}
		\hline
		& $g_0 \equiv 0$ & $g_1(x) = -|x|^{-2}$ & $g_2(x) = -|x|^{-1.5}$ & $g_3(x) = -|x|^{-1.2}$ \\ \hline
		$s = 0.3$   & $2.96$  & $3.07$ & $3.20$  & $2.69$  \\ \hline
		$s = 0.8$   & $3.84$  & $3.44$ & $2.92$  & $2.68$ \\ \hline                                     
	\end{tabular}
	\caption{ Experimental convergence rates for Example \ref{sec:ej_R}. The asymptotic behavior (see Figure \ref{fig:ordenes_R}) suggests that $\|u^{H_n}_h -u^{H_{n+1}}_h \|_{L^2(\Omega)} \lesssim H^{-c}$, for some constant $c>0$ depending of $s$ and $g$. Here we show least-squares fittings of $c$ in these examples.}
	\label{tab:ordenes}
\end{table}

\subsection{Qualitative behavior in 2D} \label{sec:ej_2d}
In order to explore the qualitative behavior of 2D solutions, we set a 2-dimensional example with $\Omega = B(0,1)$, $g(x) = -1/|x|^3$, $f \equiv 2$, and $H = 2$. In this case, $\int_\Omega f = 2\pi = - \int_{\Omega^c} g$, and thus solutions have zero mean on $\Omega$. For the implementation of \eqref{eq:weak_discrete}, we modified the code given in \cite{ABB}. We give details on the implementation of this particular example in Appendix \ref{sec:implementation}.

Results for several values of $s$ on a quasi-uniform mesh with $\Lambda_{H} = B(0,3)$ are shown in Figure \ref{fig:ej_2d_todo}. In all cases, we obtained that the discrete solutions have zero average in $\Omega$, in agreement with Remark \ref{rmk:average_uh}. The solutions exposed in Figure \ref{fig:ej_2d_todo} have different asymptotic behaviors. According to Corollary \ref{cor:strong-decay}, since for $s=0.1$ we have $g(x)|x|^{2+2s} \to 0$ as $|x| \to \infty$, solutions vanish at infinity. On the other hand, this limit blows up for $s = 0.9$ and thus $u(x) \to -\infty$ in such a case. The transition between these two behaviors happens for $s=0.5$. With the notation from Remark \ref{rmk:asintotico}, we have $\kappa = -1$ and therefore $ u(x) \to -2$ as $|x| \to \infty$ because $C_{2,0.5} = 1/2\pi$ and $|\Omega| = \pi$.

\begin{figure}[htbp]
	\centering
	\begin{tabular}{ccc}
		\includegraphics[width=0.28\linewidth]{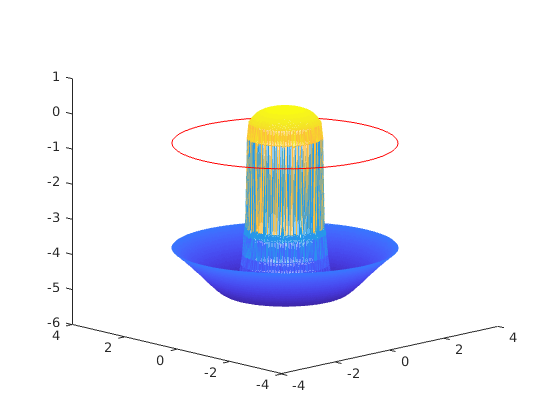}
		&
		\includegraphics[width=0.28\linewidth]{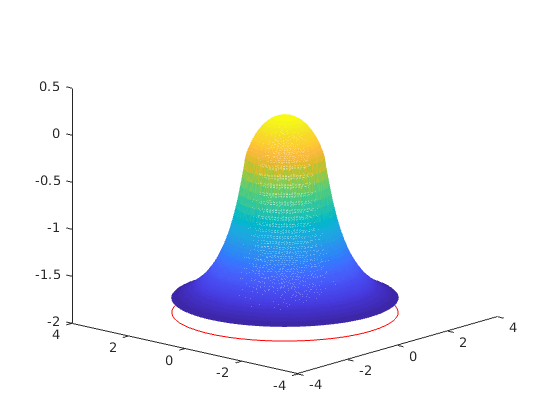}
		&
		\includegraphics[width=0.28\linewidth]{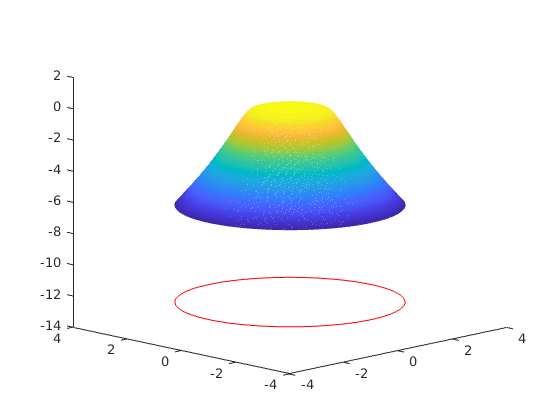}
		\\
		\subf{\includegraphics[width=0.28\linewidth]{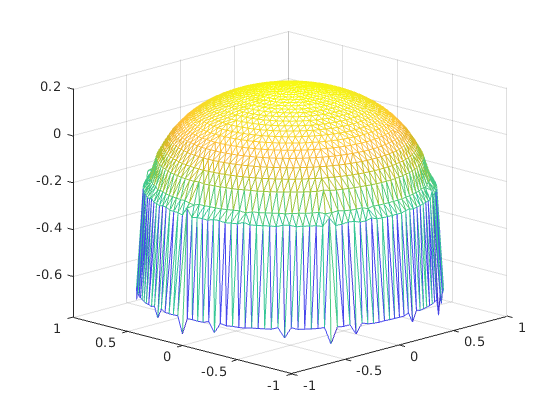}}
		{$s = 0.1$}
		&
		\subf{\includegraphics[width=0.28\linewidth]{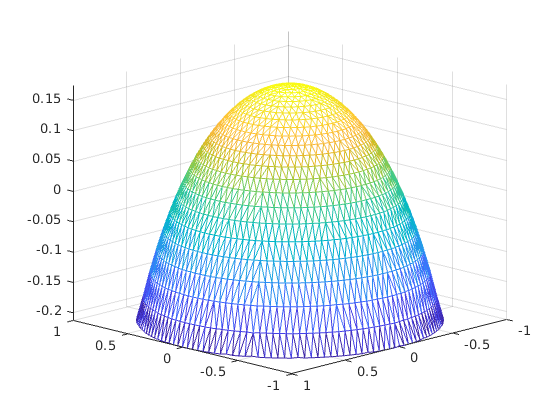}}
		{$s = 0.5$}
		&
		\subf{\includegraphics[width=0.28\linewidth]{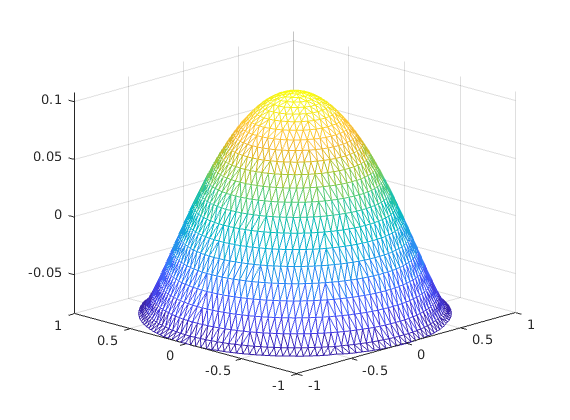}}
		{$s = 0.9$}
	\end{tabular}
	\caption{Results for the problem described in Section \ref{sec:ej_2d}, for several values of $s$ computed on a quasi-uniform mesh consisting of 32200 triangles on $\Lambda_{H}$. Top row: discrete solutions in $\Lambda_{H}$, with the value of the solution in $\Lambda_{H}^c$ represented by a red circle over $\partial \Lambda_{H}$. Bottom row: solutions in $\Omega$.} 
	\label{fig:ej_2d_todo}
\end{figure}

As an illustration of the method's ability to capture this phenomenon, Table \ref{tab:asintoticos} reports the values of $U_{N+1}= u_h \big|_{\Lambda_H^c}$ computed for three meshes $\calT_i$ ($i=1,2,3$). In all cases, $h = 5\times 10^{-2}$ in $\Omega$ and the meshes were graded in $\Omega^c$, so that the element sizes are proportional to $d(T, \Omega)^3$ for elements far away from $\Omega$. This way, the resulting computational domains $\Lambda_H$ corresponded to $H=64,216,512.$

\begin{table}[htbp]
	\begin{tabular}{|c|c|c|c|}
		\hline
		& $\calT_1$, $H = 64$ & $\calT_2$, $H = 216$ & $\calT_3$, $H=512$ \\ \hline
		$s = 0.1$ & $-0.0720$  & $-0.0283$ & $-0.0151$  \\ \hline
		$s = 0.5$ & $-2.0028$  & $-2.0029$ & $-2.0029$ \\ \hline
		$s = 0.9$ & $-158.33$  & $-419.04$ & $-835.83$  \\ \hline                                                                          
	\end{tabular}
	\caption{Values of discrete solutions at infinity for $s=0.1, 0.5, 0.9$ for meshes with different computational domains. The results are in good agreement with Corollary \ref{cor:strong-decay} and Remark \ref{rmk:asintotico}.}
	\label{tab:asintoticos}
\end{table}

\subsection{Fractional Heat Equation} \label{sec:ej_parabolico}
As a last example, we focus on the fractional heat diffusion problem with homogeneous Neumann condition \eqref{eq:Neumann_parabolico}. By combining scheme \ref{eq:weak_discrete} for the spatial discretization and a backward Euler time-stepping, we obtain the discrete problem: given $U^n_h$ ($n \in \{0,...,N-1\}$), find $U^{n+1}_h \in \V_h$ such that 
\begin{equation*} \label{eq:weak_discrete_parabolic}
\left(  \frac{U^{n+1}_h - U^{n}_h}{\delta t}  , v_h \right)_{L^2(\Omega)}  + \langle U^{n+1}_h , v_h \rangle_{\X} = 0  \quad \forall v_h \in \V_h.
\end{equation*}
Above $\delta t>0$ is a uniform time step, $\delta t = T/N$, and $U^{0}_h$ is a discretization of the initial condition $u_0$. Clearly, for every $n$, the equation above reduces to \eqref{eq:weak_discrete} with $f = U^n_h/\delta t$, $\alpha = 1/\delta t$, and $g \equiv 0$.

In our experiments we consider $\Omega = (-1 , 1)$ and $u_0(x) = I_{[-1/2,1/2]}(x)$. Numerical solutions for several values of $s$ are displayed in Figure \ref{fig:parabolico}. Additionally, according to \cite[Proposition 4.2.]{DiRoVa17}, for all $t>0$ we have $\|u - \frac{1}{|\Omega|}\int_{\Omega} u_0\|_{L^2(\Omega)} < A e^{-ct}$, for some positive constants $A$ and $c$. This exponential decay is also verified by our numerical solutions (see in Figure \ref{fig:parabolico2}).

\begin{figure}[h]
	\centering
	\begin{tabular}{ccc}
		\subf{\includegraphics[width=0.3\linewidth]{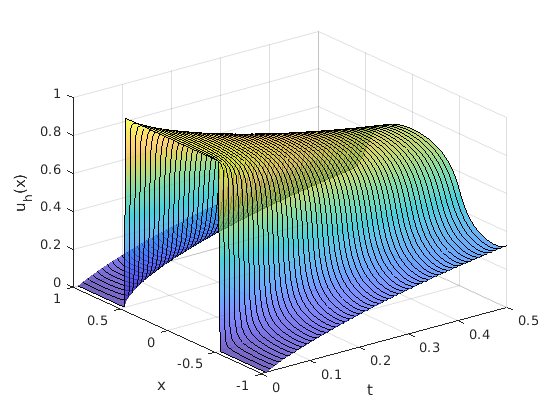}}
		{$s = 0.3$}
		&
		\subf{\includegraphics[width=0.3\linewidth]{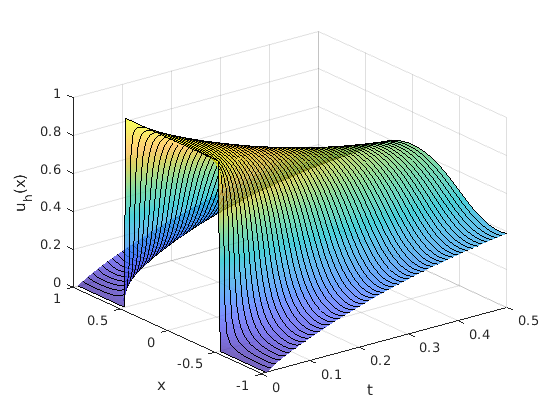}}
		{ $s = 0.5$}
		&
		\subf{\includegraphics[width=0.3\linewidth]{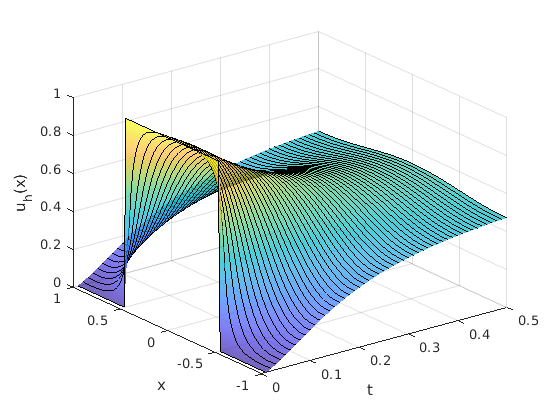}}
		{ $s = 0.8$}
		\\
	\end{tabular}
	\caption{Numerical solutions of Example \ref{sec:ej_parabolico} for several values of $s$. Here we set $\delta t = 0.01$, $h = 1/100$, and $H = 2$. As predicted in \cite{DiRoVa17}, solutions in $\Omega$ converge to the  constant $\frac{1}{|\Omega|}\int_{\Omega} u_0 = 0.5$ as $t \to \infty$.}  
	\label{fig:parabolico}
\end{figure}     

\begin{figure}[h]
	\centering
	\includegraphics[width=50mm]{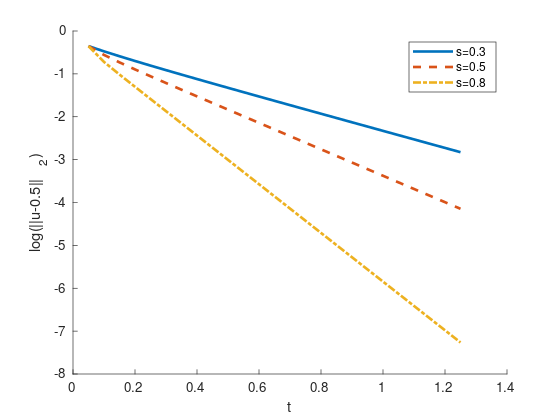}
	\caption{ Time evolution of $\log \big( \|u - \frac{1}{\Omega} \int_{\Omega} u_0 \|_{L^2(\Omega)} \big) = \log \big( \|u - 0.5 \|_{L^2(\Omega)} \big) $ in Example \ref{sec:ej_parabolico}. The linear relation between both quantities agrees with the exponential decay predicted in \cite{DiRoVa17}.} 
	\label{fig:parabolico2}
\end{figure}

\bibliographystyle{plain}
\bibliography{IMANUM-refs}

\clearpage

\appendix 

\section{Computing the right hand side in Example \ref{sec:ej_1d}} \label{app:g_construida}

In order to assemble the right hand side in Example \ref{sec:ej_1d}, we need to deal with the singularities of the flux density $g$ near $\partial \Omega$. Since we are using a regular mesh with element size $h$, this issue arises when computing   
\begin{equation}
\label{eq:singularidad}
\int_{T_i} g(x) \phii(x) \, dx,
\end{equation}
with $T_i = [1,1+h]$ or $T_i = [-1-h,-1]$. Due to the symmetry of the problem, we shall focus only on the first case. 

Indeed, consider the Lagrange basis function $\phii_j$ associated with the node $\x_j = 1$, namely, $\phii_j(x) = 1 - (x-1)/h$ for all $x \in [1,1+h]$. \fran{We recall the definitions \eqref{eq:sol_exp} and \eqref{eq:defofLaps} of the constants $c_s$ and $C_{1,s}$ respectively, so that their product equals $ c_s C_{1,s} = \frac{1}{\Gamma(1-s)\Gamma(s)}$, and} rewrite \eqref{eq:singularidad} as
\[ \begin{split}
\int_{T_i} g(x) \phii_j(x) \, dx & = \fran{-\frac{1}{\Gamma(1-s)\Gamma(s)}}\int_1^{1+h} \int^{1}_{-1} \frac{(1 - y^2)^s}{(x-y)^{1+2s}} \Big(1 - \frac{x-1}{h} \Big) \, dy \, dx  \\
&  = \fran{-\frac{2^{2s+1}h}{\Gamma(1-s)\Gamma(s)}} \int^{1}_{0}  \int^{1}_{0} \frac{\hat{y}^s (1 - \hat{y})^s}{(h\hat{x}+2\hat{y})^{1+2s}} \big(1 - \hat{x} \big) \, d\hat{y} \, d\hat{x}.  
\end{split} \]
We use that $|x-y| = (x-y)$ (because $x-y>0$ in $T_i \times \Omega$), and make the change of variables $(\hat{x}, \hat{y}) = ((x-1)/h, (1 - y)/2)$. Observing that the last integral is performed over $Q = (0,1)\times(0,1)$, we split the domain into two triangles and treat each part separately. Namely, defining
\[ \begin{split}
& D_1 := \{ (x,y) \in Q, \text{ such that } y \leq x \}, \\
& D_2 := \{ (x,y) \in Q, \text{ such that } x \leq y \}, 
\end{split} \]
we have $Q = D_1 \cup D_2$. We first analyze the integral over $D_1$.

Applying the Duffy-type transformation $T_1: Q \to D_1$, $T_1(\xi , \eta) \to (\xi, \xi \eta)$, we write
\begin{equation} \begin{split}
\label{eq:D1}
\iint_{D_1} \frac{\hat{y}^s (1 - \hat{y})^s}{(h\hat{x}+2\hat{y})^{1+2s}} \big(1 - \hat{x} \big) \, d\hat{y} \, d\hat{x} & =
\int^{1}_{0}  \int^{1}_{0} \frac{\xi^s \eta^s (1 - \xi \eta)^s}{(h\xi+2\xi \eta)^{1+2s}} \big(1 - \xi \big)\xi  \, d\eta \, d\xi \\
&   = \int^{1}_{0} \frac{ \eta^s}{ (h+2\eta)^{1+2s}} \Big( \int^{1}_{0} \frac{(1 - \xi \eta)^s (1 - \xi)}{\xi^s } \, d\xi \Big ) \, d\eta.
\end{split} \end{equation}
Let us focus on the inner singular integral. Defining
\begin{equation*}
I_1(\eta) := \int^{1}_{0} \frac{(1 - \xi \eta)^s(1 - \xi)}{\xi^s } \, d\xi,
\end{equation*}
and applying the change of variables $\xi = z^{1/(1-s)}$, we obtain
\begin{equation}
\label{eq:I_1}
I_1(\eta) = \frac{1}{1-s} \int^1_0 (1 - \eta z^{1/(1-s)} )^s (1 - z^{1/(1-s)}) \, dz.
\end{equation} 
Because the integrand is a smooth, bounded function, this expression can be accurately approximated using standard integration techniques for all $\eta \in [0,1]$, and therefore we are able to obtain good approximations of the integral in \eqref{eq:D1}.

In the same fashion, applying the transformation $T_2: Q \to D_2$, $T_2(\xi , \eta) \to (\xi \eta , \xi)$ we obtain
\begin{equation}
\label{eq:D2}
\iint_{D_2} \frac{\hat{y}^s (1 - \hat{y})^s}{(h\hat{x}-2\hat{y})^{1+2s}} \big(1 - \hat{x} \big) \, d\hat{y} \, d\hat{x} = \int^{1}_{0} \frac{ 1}{ (h\eta+2)^{1+2s}} \Big( \int^{1}_{0} \frac{(1 - \xi )^s (1 - \eta \xi)}{\xi^s } \, d\xi \Big ) \, d\eta.
\end{equation}
The function
\begin{equation}
\label{eq:I_2}
I_2(\eta) := \int^{1}_{0} \frac{(1 - \xi )^s(1 - \eta \xi)}{\xi^s } \, d\xi = \frac{1}{1-s} \int^1_0 (1 -  z^{1/(1-s)} )^s (1 - \eta z^{1/(1-s)}) \, dz,
\end{equation} 
where in the last equality we made a change of variables as in \eqref{eq:I_1}, can be accurately approximated by the same considerations as before. Finally, substituting \eqref{eq:D2} and \eqref{eq:I_1} in \eqref{eq:D1} and \eqref{eq:I_2} respectively, yields  
\begin{equation*}
\int_1^{1+h} g(x) \phii_j(x) \, dx = \fran{-C_{1,s}}c_s 2^{2s+1}h \int^{1}_{0} \frac{ \eta^s I_1(\eta)}{ (h+2\eta)^{1+2s}}  + 
\frac{ I_2 (\eta)}{ (h\eta+2)^{1+2s}}  \, d\eta,
\end{equation*}  
and standard numerical integration techniques can be applied in order to approximate the latter expression. 

The treatment of the other basis function on $T_i$, namely $\phii_j(x) = (x-1)/h$, can be handled in the same way. Following the former ideas, if we define
\begin{equation*}
I_{3}(\eta) := \int^{1}_{0} (1 - \xi \eta)^s \xi^{1-s} \, d\xi, \quad \mbox{and} \quad I_{4}(\eta) := \eta \int^{1}_{0} (1 - \xi )^s \xi^{1-s} \, d\xi,
\end{equation*}
we obtain
\begin{equation*}
\int_1^{1+h} g(x) \phii_2(x) \, dx = \fran{-C_{1,s}} c_s 2^{2s+1}h \int^{1}_{0} \frac{ \eta^s I_3(\eta)}{ (h+2\eta)^{1+2s}}  + 
\frac{ I_4 (\eta)}{ (h\eta+2)^{1+2s}}  \, d\eta.
\end{equation*}          
In this case, the functions $I_3$ and $I_4$ can be expressed in terms of beta functions: it holds that $I_3(\eta) = \eta^{s-2} B(\eta; 1-s , s )$ and $I_4(\eta) = \eta B(1-s,s)$.  

\section{Implementation details in 2D} \label{sec:implementation}

Implementing the scheme described in Section \ref{sec:FE} involves some computational challenges, such as the integration of singular functions or the computation of integrals over unbounded domains. However, many of these difficulties can be tackled using the same ideas displayed in \cite{ABB}. In this Appendix we report the modifications needed on the code given in that work in order to adapt it to our problem \footnote{A full version of this code is available on:  https://github.com/fbersetche/Finite-element-approximation-of-fractional-Neumann-problems.}.
 We shall make use of the same notation as in \cite{ABB}. To fix ideas, we restrict our attention to the setting in Example \ref{sec:ej_2d}.

\subsection{Assembling the stiffness matrix}

For the Dirichlet for the fractional Laplacian with homogeneous boundary conditions, reference \cite{ABB} uses an auxiliary domain --typically a ball-- to assemble the stiffness matrix \verb|K|. Namely, it computes interactions between basis functions supported in $\Omega$ and certain nodal basis functions supported in $\Omega^c$. We take advantage of this construction in our setting because it means we already have at hand the interactions between basis functions supported in $\Omega$ and the ones supported in the auxiliary domain $\Lambda_{H} \setminus \Omega$. 

Therefore, the missing entries in the stiffness matrix are the last row/column, that involves the interaction between the constant basis function $\phii_{N+1}$ and the remaining ones. Namely, we need to calculate
\[
K_{i, N+1} = \langle \phii_i, \phii_{N+1} \rangle_\X , \quad \mbox{for } i = 1, \ldots , N+1. 
\]
Splitting the integral in this bilinear form as in \cite[Section 3]{ABB} and using the fact that $\phii_{N+1} = \chi_{\Lambda^c_{H}}$, we realize we only need to compute, for every $T_l \subset \overline\Omega$, expressions of the form
\[
\iint_{T_l \times \Lambda^c_{H}} \frac{(\phii_i(x)-\phii_i(y))  (\phii_{N+1}(x)-\phii_{N+1}(y))}{|x-y|^{d+2s}} \, dx \, dy =  -\iint_{T_l \times \Lambda^c_{H}} \frac{\phii_i(x)}{|x-y|^{d+2s}} \, dx \, dy 
\]
for $i = 1, \ldots , N$ and 
\[ 
\iint_{T_l \times \Lambda^c_{H}} \frac{  (\phii_{N+1}(x)-\phii_{N+1}(y))^2}{|x-y|^{d+2s}} \, dx \, dy =  \iint_{T_l \times \Lambda^c_{H}} \frac{1  }{|x-y|^{d+2s}} \, dx \, dy.
\]

Because we need to compute integrals over unbounded domains, we use the function \verb|comp_quad| from \cite[Section A.5]{ABB} with a properly modified input. To this end, some modifications in the variable \verb|cphi| are needed: we compute two new auxiliary variables \verb|cphi2| and \verb|cphi3| by executing the following code after the one presented at the end of \cite[Section C.6]{ABB}: 

\begin{Verbatim}
local = cell(1,3);
local{1} = @(x,y) 1-x;
local{2} = @(x,y) x-y;
local{3} = @(x,y) y;
cphi2 = zeros(9,12);
cphi3 = zeros(9,12);
for i = 1:3
	for j = 1:3
		f1 = @(z,y) local{i}(z,y);                    
		cphi2( sub2ind([3 3], i , j) , : ) =...
		f1( p_T_12(:,1) , p_T_12(:,2) ).*w_T_12;
	end
end
for i = 1:3
	for j = 1:3
		f1 = @(z,y) -1;                    
		cphi3( sub2ind([3 3], i , j) , : ) =...
		f1( p_T_12(:,1) , p_T_12(:,2) ).*w_T_12;
	end
end
\end{Verbatim}
Above, \verb|p_T_12| and \verb|w_T_12| are the quadrature points and their respective weights (see \cite[Appendix C]{ABB}). The variables \verb|cphi2| and \verb|cphi3| play the same role as \verb|cphi|. Thus, we need to execute the former code only once and save the auxiliary variables in order to load them latter in the MATLAB workspace, before the execution of the main code. 

The main code is modified as follows.
\begin{itemize}
	\item Replace line 9 with:
	
	\begin{Verbatim}
	K  = zeros(nn+1,nn+1);
	
	\end{Verbatim} 
	
	\item Between lines 55 and 56 add the following:
	
	\begin{Verbatim}
JC = comp_quad(Bl,xl(1),yl(1),s,cphi2,R,area(l),p_I,w_I,p_T_12);
K(nodl, nn + 1) = K(nodl, nn + 1) + JC(:,1);
K(nn + 1, nodl) = K(nn + 1, nodl) + ( JC(:,1) )';
JC2 = comp_quad(Bl,xl(1),yl(1),s,cphi3,R,area(l),p_I,w_I,p_T_12);
K(nn + 1, nn + 1) = K(nn + 1, nn + 1) + JC2(1,1); 
	
	\end{Verbatim} 
	
\end{itemize}      
Note that above $\verb|R| = \mbox{diam}(\Lambda_H) = H$; we named the variable in such a way in order to be consistent with the notation from \cite{ABB}. 

\subsection{Computing the right hand side and solving the system}
Let $g$ be the Neumann datum. We need to compute 
\begin{equation*}
\int_{\R^2} \phii_i(x) g(x) \, dx, \quad \mbox{for } i = 1, \ldots, N+1.
\end{equation*} 
In Example \ref{sec:ej_2d} we have $g(x) = -1/|x|^3$. In particular, we have
\begin{equation*}
\verb|b(nn+1,1)| = \int_{\R^2} \phii_{N+1}(x) g(x) \, dx =  -\int_{\Lambda^c_{H}} \frac{1}{|x|^3} \, dx = 
-2\pi/H.
\end{equation*} 

Therefore, we modify the main code as follows to compute the right hand side in \eqref{eq:weak_discrete}. 
\begin{itemize}
	\item Define the function $f$ in $\Omega$ and $g$ in $\Lambda_{H} \setminus \Omega$, for example, after the definition of $f$. That is, overwrite line 4 with:
	\begin{Verbatim}
f = @(x,y) 2;
g = @(x,y) -1./( sqrt( x.^2 + y.^2 ) ).^3;
	
	\end{Verbatim} 
	\item Replace line 10 by:
	\begin{Verbatim}
b  = zeros(nn+1,1);
	
	\end{Verbatim}
	
	\item Comment the last two lines at the end of the main loop, and add: 
	\begin{Verbatim}
for l=nt-nt_aux+1:nt
nodl = t(l,:);
xl = p(1 , nodl); yl = p(2 , nodl);
b(nodl) = b(nodl) + fquad(area(l),xl,yl,g);
end 
b(nn+1,1) = -2*pi/R;	
	\end{Verbatim} 
	
\end{itemize}      

Besides modifying the right hand side, we need to incorporate the mass matrix and modify the system matrix accordingly. The former task is straightforward:
\begin{Verbatim}
M = zeros(nn+1,nn+1); 
for l=1:nt-nt_aux
	nodl = t(l,:);
	M(nodl,nodl) = M(nodl,nodl) + (area(l)/12).*( ones(3) + eye(3) );
end
\end{Verbatim} 

As for the second task, we set the variable $\verb|alpha| = \alpha$ as in \eqref{eq:Neumann} (here we use $\alpha=1$), and set and solve the linear system: 
\begin{Verbatim}
alpha = 1;
K = K.*cns;
uh = (K + alpha.*M)\b; 
\end{Verbatim}

Finally, we add the following lines to plot the discrete solution: 
\begin{Verbatim}
theta = 0:(2*pi)/100:2*pi;
xx = R.*cos(theta);
yy = R.*sin(theta);
zz = uh(nn+1).*ones(size(theta));
hold on
trimesh(t(1:nt , :), p(1,:),p(2,:),uh(1:end-1)); 
plot3(xx, yy, zz , '-or')
hold off
figure
trimesh(t(1:nt - nt_aux, :), p(1,:),p(2,:),uh(1:end-1)); 
\end{Verbatim}

We point out that this code returns two figures as output: the first one displays the solution in $\Lambda_{H}$, and a red circle over $\pp\Lambda_{H}$ represents the value of the numerical solution in $\Lambda^c_{H}$, as in the top row in Figure \ref{fig:ej_2d_todo}. The second figure shows the solution in $\Omega$, as in the bottom row in the same figure.    

\section*{Acknowledgements}

FMB has been supported by a PEDECIBA postdoctoral fellowship, and by ANPCyT under grant  PICT 2018 - 3017. \fran{JPB has been supported by a Fondo Vaz Ferreira grant 2019-068.}


\end{document}